\documentclass[11pt,a4paper]{article}

\usepackage{epsf,epsfig,amsfonts,amsgen,amsmath,amstext,amsbsy,amsopn,amsthm,cases,listings,color}
\usepackage{ebezier,eepic}
\usepackage{color}
\usepackage{multirow}
\usepackage{epstopdf}
\usepackage{graphicx}
\usepackage{pgf,tikz}
\usepackage{mathrsfs}
\usepackage[marginal]{footmisc}
\usepackage{enumitem}
\usepackage{tikz-cd}

\usepackage{pgf,tikz,pgfplots}
\usepackage{subfigure}
\usepackage{mathrsfs}
\usetikzlibrary{arrows}

\usepackage[margin=1in]{geometry}

\definecolor{uuuuuu}{rgb}{0.27,0.27,0.27}
\definecolor{sqsqsq}{rgb}{0.1255,0.1255,0.1255}

\newtheorem{dfn}{Definition} [section]
\newtheorem{thm}[dfn]{Theorem}
\newtheorem{lemma}[dfn]{Lemma}
\newtheorem{prop}[dfn]{Proposition}

\newtheorem{conj}[dfn]{Conjecture}
\newtheorem{claim}[dfn]{Claim}
\newtheorem{prob}[dfn]{Problem}
\newtheorem{obs}[dfn]{Observation}
\newenvironment{pf}{\noindent{\bf Proof}}

\def\lc{\left\lceil}

\def\rc{\right\rceil}

\def\lf{\left\lfloor}

\def\rf{\right\rfloor}
\begin{document}
\title{\bf\Large A hypergraph Tur\'{a}n problem with no stability}

\date{\today}

\author{Xizhi Liu\thanks{Department of Mathematics, Statistics, and Computer Science, University of Illinois, Chicago, IL, 60607 USA. Email: xliu246@uic.edu.}
\and
Dhruv Mubayi\thanks{Department of Mathematics, Statistics, and Computer Science, University of Illinois, Chicago, IL, 60607 USA. Email: mubayi@uic.edu.
Research partially supported by NSF award DMS-1763317.}
}
\maketitle
\begin{abstract}
A fundamental barrier in extremal hypergraph theory is the presence of many near-extremal constructions with very different structure.
Indeed, the classical constructions due to Kostochka imply that the notorious extremal problem for the tetrahedron
exhibits this phenomenon assuming Tur\'an's conjecture.

Our main result is to construct a finite family of triple systems $\mathcal{M}$, determine its Tur\'{a}n number,
and prove that there are two near-extremal $\mathcal{M}$-free constructions that are far from each other in edit-distance.
This is the first extremal result for a hypergraph family that fails to have a corresponding stability theorem.
\end{abstract}

\section{Introduction}
Let $r \ge 3$ and $\mathcal{F}$ be a family of $r$-uniform graphs (henceforth $r$-graphs).
An $r$-graph $\mathcal{H}$ is $\mathcal{F}$-free if it contains no member of $\mathcal{F}$ as a subgraph.
The Tur\'{a}n number $\textrm{ex}(n,\mathcal{F})$ of $\mathcal{F}$ is the maximum number of edges in an $\mathcal{F}$-free $r$-graph on $n$ vertices.
The Tur\'{a}n density $\pi(\mathcal{F})$
of $\mathcal{F}$ is defined as $\pi(\mathcal{F}) := \lim_{n \to \infty} \textrm{ex}(n,\mathcal{F}) / \binom{n}{r}$.
The study of $\textrm{ex}(n,\mathcal{F})$ is perhaps the central topic in extremal graph and hypergraph theory.

Much is known about $\textrm{ex}(n,\mathcal{F})$ when $r=2$ and one of the most famous results in this regard is Tur\'{a}n's theorem,
which states that for $\ell \ge 2$ the Tur\'{a}n number ${\rm ex}(n,K_{\ell+1})$ is uniquely achieved by
$T(n,\ell)$ which is the $\ell$-partite graph on $n$ vertices with the maximum number of edges.

For $\ell > r \ge 3$, let $K_{\ell}^{r}$ be the complete $r$-graph on $\ell$ vertices.
Extending Tur\'{a}n's theorem to hypergraphs (i.e. $r\ge 3$) is a major problem.
Indeed, the problem of determining $\pi(K_{\ell}^{r})$ was raised by Tur\'{a}n \cite{TU41} and is still wide open.
Erd\H{o}s offered $\$ 500$ for the determination of any $\pi(K_{\ell}^{r})$ with $\ell > r \ge 3$
and $\$ 1000$ for the determination of all $\pi(K_{\ell}^{r})$ with $\ell > r \ge 3$.

\begin{conj}[Tur\'{a}n \cite{TU41}]\label{Turan-conj-r-gp}
$\pi(K_{\ell}^{3}) = 1- \left( \frac{2}{\ell -1} \right)^2$.
\end{conj}

The case $\ell = 4$ above, which states that $\pi(K_{4}^{3}) = 5/9$ has generated a lot of interest and activity over the years.
Many constructions (e.g. see \cite{BR83,KO82,FO88}) are known to achieve the value in Conjecture \ref{Turan-conj-r-gp} for $\ell = 4$,
and that is perhaps one of the reasons why it is so challenging.
On the other hand, successively better upper bounds for $\pi(K_{4}^{3})$ were obtained by
de Caen \cite{DC88}, Giraud (see \cite{CL99}), Chung and Lu  \cite{CL99}, and Razborov \cite{RA10}.
The current record is $\pi(K_{4}^{3}) \le 0.561666$, which was obtained by Razborov  \cite{RA10} using flag algebra machinery.

Many families $\mathcal{F}$ have the property that there is a unique $\mathcal{F}$-free graph (or hypergraph)
$\mathcal{G}$ on $n$ vertices achieving $\textrm{ex}(n,\mathcal{F})$,
and moreover, any $\mathcal{F}$-free graph (or hypergraph) $\mathcal{H}$ of size close to $\textrm{ex}(n,\mathcal{F})$ can be transformed to $\mathcal{G}$
by deleting and adding very few edges.
Such a property is called the stability of $\mathcal{F}$.
The first stability theorem was proved independently by Erd\H{o}s and Simonovits \cite{SI68}.

\begin{thm}[Erd\H{o}s-Simonovits \cite{SI68}]\label{Erdos-Simo-stability}
Let $\ell \ge 2$. Then for every $\delta > 0$ there exists an $\epsilon > 0$ and $n_0$ such that the following statement holds for all $n \ge n_0$.
Every $K_{\ell+1}$-free graph on $n$ vertices with at least $(1-\epsilon) t(n,\ell)$ edges can be transformed to $T(n,\ell)$
by deleting and adding at most $\delta n^2$ edges.
\end{thm}

The stability phenomenon has been used to determine $\textrm{ex}(n,\mathcal{F})$ exactly in many cases.
It was first used by Simonovits in \cite{SI68} to determine $\textrm{ex}(n, \mathcal{F})$ exactly for all color-critical graphs and large $n$,
and then by several authors (e.g. see \cite{DF00,FV13,FPS05,ZS05,KS05a,KS05b,MP07,MPS11,PI13}) to prove exact results for hypergraphs.

However, there are many Tur\'{a}n problems for hypergraphs that (perhaps) do not have the stability property.
The example $K_{4}^{3}$ we mentioned before was shown to have exponentially many extremal constructions  in the number of vertices (see Kostochka~\cite{KO82} and Brown~\cite{BR83}).
We will prove (Proposition \ref{prop-stability-number-K43}) that these constructions can be used to show that $K_{4}^{3}$ does not have the stability property
(assuming Conjecture \ref{Turan-conj-r-gp} is true).
For $K_{\ell}^{3}$ with $\ell \ge 5$, different near-extremal constructions were given by Sidorenko \cite{AS95},
and Keevash and the second author \cite{KE11}.
Although we do not provide the details, these  also show that $K_{\ell}^{3}$ does not have stability (assuming Conjecture \ref{Turan-conj-r-gp} is true).

Another example in which the forbidden family consists of infinitely many $3$-graphs arises from the following longstanding
conjecture due to Erd\H{o}s and S\'{o}s.
For an $r$-graph $\mathcal{H}$ we use $V(\mathcal{H})$ to denote its vertex set and $v(\mathcal{H})=|V(\mathcal{H})|$.
The link of $v \in V(\mathcal{H})$ is
\begin{align}
L_{\mathcal{H}}(v) := \left\{ A\in \binom{V(\mathcal{H})}{r-1}: \{v\} \cup A \in \mathcal{H} \right\}. \notag
\end{align}
We will omit the subscript $\mathcal{H}$ if it is clear from context.

\begin{conj}[Erd\H{o}s-S\'{o}s, see \cite{FF84}]\label{bipart-link-conj}
Let $\mathcal{H}$ be a $3$-graph with $n$-vertices.
If $L(v)$ is bipartite for all $v \in V(\mathcal{H})$,
then $|\mathcal{H}| \le (1/4 + o(1))\binom{n}{3}$.
\end{conj}

If Conjecture \ref{bipart-link-conj} is true,
then it also does  not have the stability property as there are several different near-extremal constructions given in \cite{MPS11}.

The absence of stability seems to be a fundamental barrier in determining the Tur\'{a}n numbers of some families.
Indeed, the Tur\'{a}n numbers of the examples we presented above are not known, even asymptotically,
and in fact, no Tur\'{a}n number of a family without the stability property has been determined.

This paper provides the first such example.
We construct a family $\mathcal{M}$ of $3$-graphs,
prove that $\mathcal{M}$ does not have the stability property,
and  determine $\pi(\mathcal{M})$,
and even ${\rm ex}(n,\mathcal{M})$ for infinitely many $n$ (Theorems \ref{thm-Turan-num-M} and \ref{thm-2-stable}).

The present paper has a slightly similar flavor as \cite{Pi14} in the sense that
we will define the extremal hypergraphs $\mathcal{G}_{n}^{1}$ and $\mathcal{G}_{n}^2$ first,
and then define the forbidden family $\mathcal{M}$,
which is a suitably chosen family based on $\mathcal{G}_{n}^{1}$ and $\mathcal{G}_{n}^2$.

In order to state our results formally, we need some definitions.

\begin{dfn}\label{dfn-trans-num}
Let $r \ge 2$ and $\mathcal{H}$ be an $r$-graph.
The transversal number of $\mathcal{H}$ is
\[
\tau(\mathcal{H}) := \min\left\{ |S| :S \subset V(\mathcal{H}) \text{ such that } S\cap E \neq \emptyset \text{ for all } E \in \mathcal{H} \right\}.
\]
We set $\tau(\mathcal{H}) = 0$ if $\mathcal{H}$ is an empty graph.
\end{dfn}

Let $\ell \ge r\ge 2$ and $\mathcal{K}_{\ell+1}^{r}$ be the collection of all $r$-graphs $F$
on at most $(\ell+1)+(r-2)\binom{\ell+1}{2}$ vertices such that
for some $({\ell}+1)$-set $S$, which will be called the core of $F$, every pair $\{u,v\}\subset S$ is covered by an edge in $F$
\footnote{The original definition of $\mathcal{K}_{\ell+1}^{r}$ in \cite{MU06}
requires that $F$ has at most $\binom{{\ell}+1}{2}$ edges.
The new definition we used here will make our proofs simpler.
Notice that $\mathcal{K}_{\ell+1}^{r}$ is a finite family in both definitions.}.
Let $V_{1} \cup \cdots \cup V_{\ell}$ be a partition of $[n] := \{1,\ldots,n\}$
with each $V_{i}$ of size either $\lf n/\ell \rf$ or $\lc n/\ell \rc$.
The generalized Tur\'{a}n graph $T_{r}(n,\ell)$ is the collection of all $r$-subsets of $[n]$ that have
at most one vertex in each $V_{i}$. Let $t_{r}(n,\ell) = |T_{r}(n,\ell)|$.
It was shown by the second author \cite{MU06} that $\textrm{ex}(n,\mathcal{K}_{\ell+1}^{r}) = t_{r}(n,\ell)$ and $T_{r}(n,\ell)$ is the unique
$\mathcal{K}_{\ell+1}^{r}$-free $r$-graph on $n$ vertices with exactly $t_{r}(n,\ell)$ edges.

Suppose that $\mathcal{T}$ is an $r$-graph on $s$ vertices and $t = (t_1,\ldots,t_s)$ with each $t_i$ a positive integer.
Then the blowup $\mathcal{T}(t)$ of $\mathcal{T}$ is obtained from $\mathcal{T}$ by replacing each vertex $i$ by a set of size $t_i$,
and replacing every edge in $\mathcal{T}$ by the corresponding complete $r$-partite $r$-graph.

An $r$-graph $\mathcal{S}$ is a star if all edges in $\mathcal{S}$ contain a fixed vertex $v$, which is called the center of $\mathcal{S}$.

\begin{dfn}\label{dfn-G1-G2}
Let $|A| = \lf n/3 \rf$ and $|B| = \lc 2n/3 \rc$ with $A \cap B = \emptyset$.
Define
\[
\mathcal{G}^{1}_{n} := \left\{ abb': a\in A \text{ and } \{b,b'\}\subset B \right\}.
\]
Let $\mathcal{G}^{2}_{6}$ be the $3$-graph with vertex set $[6]$ whose complement is
\[
\overline{\mathcal{G}^{2}_{6}} := \{123, 126, 345, 456\}.
\]
For $n > 6$ let $\mathcal{G}^{2}_{n}$ be a $3$-graph on $n$ vertices
which is a blowup of $\mathcal{G}^{2}_{6}$ with the maximum number of edges.
\end{dfn}

{\bf Remarks.}
\begin{itemize}
\item
Notice that $\mathcal{G}^{1}_{n}$ is a (unbalanced) blowup of a star.
\item
Simple calculations show that each part in $\mathcal{G}^{2}_{n}$ has size either $\lf n/6 \rf$ or $\lc n/6 \rc$.
\item
For $i =1,2$, let $g_{i}(n) = |\mathcal{G}_{n}^{i}|$.
Then $\lim_{n \to \infty} g_{i}(n) / n^3 = 2/27$.
\end{itemize}

\begin{figure}[htbp]
\centering
\subfigure[The $3$-graph $\mathcal{G}^1_{n}$.]{
\begin{minipage}[t]{0.45\linewidth}
\centering
\begin{tikzpicture}[xscale=3.5,yscale=3.5]
    \node (a) at (-0.4,0.2) {};
    \node (b) at (0.4,0.1) {};
    \node (b') at (0.4,0.3) {};
    \node (c1) at (0,0) {};
    \node (c2) at (0,-0.1) {};
    \node (c3) at (0,-0.2) {};
    \fill (a) circle (0.015) node [right] {$a$};
    \fill (b) circle (0.015) node [left] {$b$};
    \fill (b') circle (0.015) node [left] {$b'$};
    \fill (c1) circle (0.012);
    \fill (c2) circle (0.012);
    \fill (c3) circle (0.012);
    \draw[rotate=0,line width=0.8pt] (-0.4,0) ellipse [x radius=0.15, y radius=0.35];
    \draw[rotate=0,line width=0.8pt] (0.4,0) ellipse [x radius=0.2, y radius=0.7];
    \draw[line width=0.8pt,color=sqsqsq,fill=sqsqsq,fill opacity=0.15]
        (-0.4-0.05,0.2) to [out = 90, in = 180] (-0.4,0.2+0.05) to [out = 0, in = 180]
        (0.4,0.3+0.05) to [out = 0, in = 0] (0.4,0.1-0.05) to [out = 180, in = 0] (-0.4,0.2-0.05)
        to [out = 180, in = 270] (-0.4-0.05,0.2);
    \node at (-0.4,-0.35-0.07) {$A$};
    \node at (0.4,-0.7-0.07) {$B$};
\end{tikzpicture}
\end{minipage}
}
\subfigure[The complement of $\mathcal{G}^2_6$.]{
\begin{minipage}[t]{0.45\linewidth}
\centering
\begin{tikzpicture}[xscale=2.5,yscale=2.5]
    \node (a1) at (-0.5,0.5) {};
    \node (a2) at (0.5,0.5) {};
    \node (a3) at (0.5,-0.5) {};
    \node (a4) at (-0.5,-0.5) {};
    \node (b1) at (-1,0) {};
    \node (b2) at (1,0) {};
    \fill (a1) circle (0.02) node [right] {$1$};
    \fill (a2) circle (0.02) node [left] {$2$};
    \fill (a3) circle (0.02) node [left] {$4$};
    \fill (a4) circle (0.02) node [right] {$5$};
    \fill (b1) circle (0.02) node at (-1-0.2,0) {$6$};
    \fill (b2) circle (0.02) node at (1+0.2,0) {$3$};
    \draw[line width=0.8pt,color=sqsqsq,fill=sqsqsq,fill opacity=0.15]
        (0.5+0.1,0.5) to [out = 90, in = 0] (-0.5,0.5+0.1) to [out = 180, in = 180] (-1,0-0.1)
        to [out = 0, in = 180] (-0.5,0.5-0.1) to [out = 0, in = 270] (0.5+0.1,0.5);
    \draw[line width=0.8pt,color=sqsqsq,fill=sqsqsq,fill opacity=0.15]
        (-0.5-0.1,0.5) to [out = 90, in = 180] (0.5,0.5+0.1) to [out = 0, in = 0] (1,0-0.1)
        to [out = 180, in = 0] (0.5,0.5-0.1) to [out = 180, in = 270] (-0.5-0.1,0.5);
    \draw[line width=0.8pt,color=sqsqsq,fill=sqsqsq,fill opacity=0.15]
        (0.5+0.1,-0.5) to [out = 270, in = 0] (-0.5,-0.5-0.1) to [out = 180, in = 180] (-1,0+0.1)
        to [out = 0, in = 180] (-0.5,-0.5+0.1) to [out = 0, in = 90] (0.5+0.1,-0.5);
    \draw[line width=0.8pt,color=sqsqsq,fill=sqsqsq,fill opacity=0.15]
        (-0.5-0.1,-0.5) to [out = 270, in = 180] (0.5,-0.5-0.1) to [out = 0, in = 0] (1,0+0.1)
        to [out = 180, in = 0] (0.5,-0.5+0.1) to [out = 180, in = 90] (-0.5-0.1,-0.5);
\end{tikzpicture}
\end{minipage}
}
\centering
\caption{$\mathcal{G}^1$ and $\overline{\mathcal{G}^2_6}$.}
\end{figure}
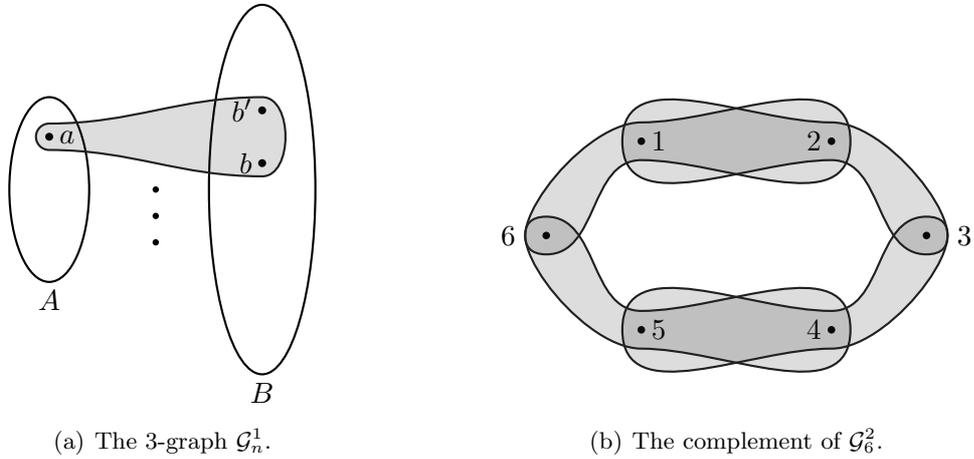

\begin{dfn}\label{dfn-M}
The family $\mathcal{M}$ is the union of the following three finite families.
\begin{enumerate}[label=(\alph*)]
\item
$M_1$ is the set containing the complete $3$-graph on five vertices with one edge removed, $M_1 = \left\{ K_{5}^{3-} \right\}$.
\item
$M_2$ is the collection of all $3$-graphs in $\mathcal{K}_{7}^{3}$ with a core  whose induced subgraph  has transversal number at least two.
\item
$M_3$ is the collection of all $3$-graphs  $F\in \mathcal{K}_{6}^{3}$ such that both $F\not\subset \mathcal{G}_{n}^{1}$ and
$F\not\subset \mathcal{G}_{n}^{2}$ for all positive integers $n$.
\end{enumerate}
\end{dfn}

Our first result is about the Tur\'{a}n number of $\mathcal{M}$.
\begin{thm}\label{thm-Turan-num-M}
The inequality $\textrm{ex}(n,\mathcal{M}) \le 2n^3/27$ holds for all positive integers $n$.
Moreover, equality holds whenever $n$ is a multiple of six.
\end{thm}

For an $r$-graph $\mathcal{H}$
the shadow of $\mathcal{H}$ is
\[
\partial \mathcal{H} := \left\{ A \in \binom{V(\mathcal{H})}{r-1}: \exists B \in \mathcal{H} \text{ such that } A\subset B \right\}.
\]
Note that both $\mathcal{G}_{n}^{1}$ and $\mathcal{G}_{n}^{2}$ are $\mathcal{M}$-free and
$g_1(n) \sim g_2(n) \sim 2n^3/27$.
Moreover, it is easy to see that transforming $\mathcal{G}_{n}^{1}$ to $\mathcal{G}_{n}^{2}$
requires us to delete and add $\Omega(n^3)$ edges.
Indeed,
$\partial\mathcal{G}_{n}^{1}$ contains a clique on $\lfloor 2n/3 \rfloor$ vertices,
whereas $\partial\mathcal{G}_{n}^{2}$ has clique number six.
By Tur\'{a}n's theorem, one must thus delete strictly more that $\left(1-\pi(K_7)\right)\binom{\lfloor 2n/3 \rfloor}{2} = \Omega(n^2)$
edges from $\partial\mathcal{G}_{n}^{2}$ to obtain a copy of $\partial\mathcal{G}_{n}^{2}$.
Since every edge in $\partial \mathcal{G}_{n}^{1}$
is covered by $\Omega(n)$ edges in $\mathcal{G}_{n}^{1}$,
we need to remove at least $\Omega(n^3)$ edges from
$\mathcal{G}_{n}^{1}$ before getting $\mathcal{G}_{n}^{2}$.
So this proves that $\mathcal{M}$ does not have the stability property (in the sense of Theorem \ref{Erdos-Simo-stability}).

In order to capture the structural property of families with more than one near-extremal structure,
the second author $\cite{MU07}$ introduced the concept of $t$-stable families, which is an extension of the classical definition of stability.
\begin{dfn}[$t$-stable, see \cite{MU07}]\label{dfn-t-stable}
Let $r \ge 2$ and $t \ge 1$ and $\mathcal{F}$ be a family of $r$-graphs.
Then $\mathcal{F}$ is $t$-stable if there exists $m_0$ and $r$-graphs $\mathcal{H}_{m}^{1},\ldots,\mathcal{H}_{m}^{t}$ on $m$ vertices
for every $m \ge m_0$ such that the following holds.
For every $\delta >0$, there exists $\epsilon > 0$ and $n_0$ such that for all $n \ge n_0$,
if $\mathcal{H}$ is an $\mathcal{F}$-free $r$-graph on $n$ vertices with
\[
|\mathcal{H}| > (1-\epsilon) {\rm ex}(n,\mathcal{F}),
\]
then $\mathcal{H}$ can be transformed to some $\mathcal{H}_{n}^{i}$ by adding and removing at most $\delta |\mathcal{H}|$ edges.
Say $\mathcal{F}$ is stable if it is $1$-stable.
\end{dfn}

According to the definition, if a family $\mathcal{F}$ is $t$-stable, then it is $s$-stable for all $s > t$,
since we can get $s$ constructions $\mathcal{G}_{m}^{1}, \ldots, \mathcal{G}_{m}^{s}$ by simply duplicating some
$\mathcal{G}_{m}^{i} \in \{\mathcal{G}_{m}^{1}, \ldots, \mathcal{G}_{m}^{t}\}$  into $s-t$ copies of itself.
However, we are actually interested in the minimum integer $t$ such that $\mathcal{F}$ is $t$-stable.

\begin{dfn}[Stability number]\label{dfn-stability-number}
Let $\mathcal{F}$ be a family of $r$-graphs.
The stability number of $\mathcal{F}$, denoted by $\xi(\mathcal{F})$, is the minimum integer $t$ such that $\mathcal{F}$ is $t$-stable.
If there is no such integer $t$, then we let $\xi(\mathcal{F}) = \infty$.
\end{dfn}

Note that in \cite{PI08} Pikhurko also defined $t$-stable families independently.
Roughly speaking, a family $\mathcal{F}$ is $t$-stable if there exist $t$ near-extremal constructions,
and every $\mathcal{F}$-free graph (or hypergraph) of size close to $\textrm{ex}(n,\mathcal{F})$
is structurally close to one of these near-extremal constructions.
Although the concept of $t$-stable families was raised over a decade ago,
no example of $t$-stable families are known for any $t \ge 2$ before this work.
However, if we assume that Tur\'{a}n's conjecture is true, then
the following result shows that the stability number of $K_{4}^{3}$ is infinity.

\begin{prop}\label{prop-stability-number-K43}
If Conjecture \ref{Turan-conj-r-gp} is true, then $\xi(K_{4}^{3}) = \infty$.
\end{prop}

Our next result gives further detail about near-extremal $\mathcal{M}$-free constructions by showing that $\mathcal{M}$
is $2$-stable with respect to $\mathcal{G}^{1}_{n}$ and $\mathcal{G}^{2}_{n}$.
More precisely, it shows that $\xi(\mathcal{M}) = 2$.

\begin{dfn}\label{dfn-semibi-G2-color}
Let $\mathcal{H}$ be a $3$-graph.
Then $\mathcal{H}$ is called semibipartite if $V(\mathcal{H})$ has a partition $A \cup B$
such that $|E \cap A| = 1$ and $|E \cap B| = 2$ for all $E \in \mathcal{H}$,
and $\mathcal{H}$ is called $\mathcal{G}^2_{6}$-colorable if it is a subgraph of a blowup of $\mathcal{G}_{6}^{2}$.
\end{dfn}

With some calculations one can get the following observation.
\begin{obs}\label{obs-size-G1-G2}
Let $\mathcal{H}$ be a $3$-graph on $n$-vertices.
If $\mathcal{H}$ is semibipartite, then $|\mathcal{H}| \le g_1(n)$.
If $\mathcal{H}$ is $\mathcal{G}^2_{6}$-colorable, then $|\mathcal{H}| \le g_2(n)$.
\end{obs}

\begin{thm}[$2$-stability]\label{thm-2-stable}
For every $\delta > 0$ there exists $\epsilon >0$ and $n_0$ such that the following holds for all $n \ge n_0$.
Every $\mathcal{M}$-free $3$-graph on $n$ vertices with at least $2n^3 /27 - \epsilon n^3$ edges can be transformed to a $3$-graph
that is either semibipartite or $\mathcal{G}^2_6$-colorable by removing at most $\delta n$ vertices.
In other words, $\xi(\mathcal{M}) = 2$.
\end{thm}

Note that Theorem \ref{thm-2-stable} is stronger than the requirement in the definition of $2$-stability
since removing at most $\delta n$ vertices implies that
the number of edges removed is at most $\delta n^3$ but not vice versa.

Let $\mathcal{H}$ be an $r$-graph on $n$ vertices.
The edge density of $\mathcal{H}$ is $d(\mathcal{H}):= |\mathcal{H}|/\binom{n}{r}$
and the shadow density of $\mathcal{H}$ is $d(\partial\mathcal{H}):= |\partial\mathcal{H}|/\binom{n}{r-1}$.
The feasible region $\Omega(\mathcal{F})$ of $\mathcal{F}$
is the set of points $(x,y)\in [0,1]^2$ such that there exists a sequence of $\mathcal{F}$-free $r$-graphs
$\left( \mathcal{H}_{k}\right)_{k=1}^{\infty}$ with $\lim_{k \to \infty}v(\mathcal{H}_{k}) = \infty$,
$\lim_{k \to \infty}d(\partial\mathcal{H}_{k}) = x$ and $\lim_{k \to \infty}d(\mathcal{H}_{k}) = y$.
We introduced this notion recently in \cite{LM19A}  to understand the extremal properties of
$\mathcal{F}$-free hypergraphs  beyond just the determination of $\pi(\mathcal{F})$ (it unifies and generalizes many classical problems).
In particular, we proved that $\Omega(F)$ is completely determined by a left-continuous almost everywhere differentiable function
$g(\mathcal{F}) : {\rm proj}\Omega(\mathcal{F}) \to [0,1]$,
where
\[
{\rm proj}\Omega(\mathcal{F}) = \left\{ x : \text{$\exists y \in [0,1]$ such that $(x,y) \in \Omega(\mathcal{F})$} \right\},
\]
and
\[
g({\mathcal{F}},x) = \max\left\{y: (x,y) \in \Omega(\mathcal{F}) \right\}, \text{ for all } x \in {\rm proj}\Omega(\mathcal{F}).
\]
Theorem \ref{thm-Turan-num-M} together with Theorem \ref{thm-2-stable} yield the following result.

\begin{thm}\label{thm-2-global-max}
The set ${\rm proj}\Omega(\mathcal{M}) = [0,1]$, and $g({\mathcal{M}},x) \le 4/9$ for all $x \in [0,1]$.
Moreover, $g({\mathcal{M}},x) = 4/9$ iff $x \in \{5/6,8/9\}$.
\end{thm}

In words, Theorem \ref{thm-2-global-max} says that $\mathcal{M}$-free $3$-graphs can have any possible shadow density
but the edge density is maximized for exactly two values of the shadow densities.

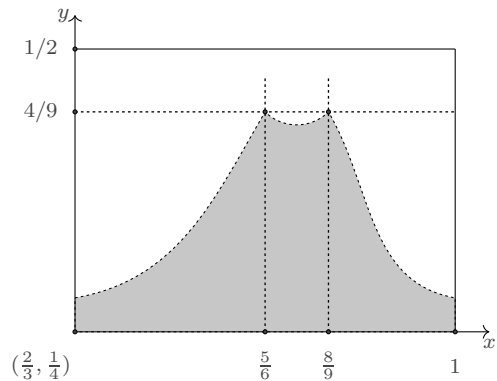
\begin{figure}[htbp]
\centering
\begin{tikzpicture}[xscale=15,yscale=15]
\draw [line width=0.5pt,dash pattern=on 1pt off 1.2pt,domain=2/3:1] plot(\x,4/9);
\draw [line width=0.5pt,dash pattern=on 1pt off 1.2pt] (5/6,1/4)--(5/6,4/9+0.03);
\draw [line width=0.5pt,dash pattern=on 1pt off 1.2pt] (8/9,1/4)--(8/9,4/9+0.03);
\draw [->] (2/3,1/4)--(1+0.03,1/4);
\draw [->] (2/3,1/4)--(2/3,1/2+0.03);
\draw (1,1/4)--(1,1/2);
\draw (2/3,1/2)--(1,1/2);
\draw[fill=sqsqsq,fill opacity=0.25, dash pattern=on 1pt off 1.2pt]
(2/3,1/4+0.03) to [out = 10, in = 240] (5/6,4/9) to [out = 315, in = 225] (8/9,4/9) to [out = 300, in = 170] (1,1/4+0.03)
-- (1,1/4) -- (2/3,1/4) -- (2/3,1/4+ 0.03);
\begin{scriptsize}
\draw [fill=uuuuuu] (5/6,4/9) circle (0.05pt);
\draw [fill=uuuuuu] (8/9,4/9) circle (0.05pt);
\draw [fill=uuuuuu] (2/3,1/4) circle (0.05pt);
\draw[color=uuuuuu] (2/3-0.03,1/4-0.03) node {$(\frac{2}{3}, \frac{1}{4})$};
\draw [fill=uuuuuu] (2/3,1/2) circle (0.05pt);
\draw[color=uuuuuu] (2/3-0.03,1/2) node {$1/2$};
\draw [fill=uuuuuu] (2/3,4/9) circle (0.05pt);
\draw[color=uuuuuu] (2/3-0.03,4/9) node {$4/9$};
\draw [fill=uuuuuu] (5/6,1/4) circle (0.05pt);
\draw[color=uuuuuu] (5/6,1/4-0.03) node {$\frac{5}{6}$};
\draw [fill=uuuuuu] (8/9,1/4) circle (0.05pt);
\draw[color=uuuuuu] (8/9,1/4-0.03) node {$\frac{8}{9}$};
\draw [fill=uuuuuu] (1,1/4) circle (0.05pt);
\draw[color=uuuuuu] (1,1/4-0.03) node {$1$};
\draw[color=uuuuuu] (1+0.03,1/4-0.01) node {$x$};
\draw[color=uuuuuu] (2/3-0.01,1/2+0.03) node {$y$};
\end{scriptsize}
\end{tikzpicture}
\caption{$g(\mathcal{M})$ has exactly two global maxima by Theorem \ref{thm-2-global-max}.}
\end{figure}

This paper is organized as follows.
In Section~\ref{SEC:proof-K43-is-infinity-stable} we prove Proposition~\ref{prop-stability-number-K43}.
In Section~\ref{SEC:prelim} we present some preliminary definitions and lemmas.
In Section~\ref{SEC:proof-Turan-number-M} we prove Theorem \ref{thm-Turan-num-M},
and in Section~\ref{SEC:proof-stability-M} we prove Theorem \ref{thm-2-stable}.
In Section~\ref{SEC:feasible-region-2-global-max} we prove Theorem~\ref{thm-2-global-max}.
We include some open problems in Section~\ref{SEC:open-problem}.

\section{Proof of Proposition \ref{prop-stability-number-K43}}\label{SEC:proof-K43-is-infinity-stable}
In this section we prove Proposition \ref{prop-stability-number-K43}.

In \cite{KO82}, Kostochka constructed exponentially many
$K_{4}^{3}$-free $3$-graphs on $n$ vertices whose edge densities achieve $5/9$;
one of his constructions is as follows (also see Brown \cite{BR83}).

Suppose that $n$ is a multiple of $3$ and $[n]=V_1\cup V_2\cup V_3$ is a balanced partition of $[n]$.
For $i \in [3]$ let $V_{i} = V_{1,i} \cup V_{2,i}$, where $|V_{1,i}| = m$ and $|V_{2,i}| = n/3-m$.
Let $\mathcal{G}(n,m)$ be 3-graph with vertex set $[n]$ whose edge set is the union of the following sets:
\begin{enumerate}
    \item $\{ abc: a,b,c \in V_{i,j}, i,j\in[3] \}$.
    \item $\{ abc: a,b \in V_{1,j}, c\in V_{2,j}, j\in [3]\}$.
    \item $\{ abc: a,b \in V_{1,j}, c\in V_{j+1}, j\in [3]\}$.
    \item $\{ abc: a,b \in V_{2,j}, c\in V_{1,j+1}, j\in [3] \}$.
    \item $\{ abc: a,b \in V_{2,j}, c\in V_{j-1}, j\in [3]\}$.
    \item $\{ abc: a\in V_{1,j}, b\in V_{2,j}, c\in V_{j+1}, j\in [3]\}$.
\end{enumerate}
The complement of $\mathcal{G}(n,m)$ is $K_{4}^{3}$-free
and $$|\mathcal{G}(n,m)| = \frac{n(n-3)(2n-3)}{27}$$ for all $0 \le m \le n/3$ (see~\cite{KO82} for details).

The {\em edit distance} {\rm ed}$(H_1, H_2)$ between two $r$-graphs $H_1$ and $H_2$ on the same vertex set is the minimum number of edges that need to be deleted and added to $H_1$ to obtain an isomorphic copy of $H_2$. The idea in the proof of Proposition \ref{prop-stability-number-K43} is
to choose a family $\mathcal{G}$ of $K_{4}^{3}$-free $3$-graphs from Kostochka's constructions
described above and show that every pair $\{\mathcal{G}^1,\mathcal{G}^2\} \subset \mathcal{G}$ has large edit distance. We begin with the following useful lemma. Write $K_4^{3-}$ for the unique 3-graph with four vertices and three edges.

\begin{lemma} \label{lemk43-} Let $\varphi(n,m)$ denote the number of induced $K_{4}^{3-}$ in $\mathcal{G}(n,m)$. Then
\begin{equation} \label{eqnk43-}\varphi(n,m) = \frac{1}{6}m^2(n-3m)(n-3m-3).\end{equation}
\end{lemma}
\begin{proof}
It is easy to see that $S \subset [n]$ induces a copy of $K_{4}^{3-}$ in $\mathcal{G}(n,m)$ iff
$|S\cap V_{1,j}| = |S\cap V_{1,j+1}| = 1$ and $|S \cap V_{2,j}|=2$ for some $j\in[3]$.
Therefore,
$$
    \varphi(n,m) = 3m^2\binom{n/3-m}{2}
                  $$
                  and expanding the binomial coefficient yields (\ref{eqnk43-}).\end{proof}

\begin{proof}[Proof of Proposition \ref{prop-stability-number-K43}]
We are going to show that for every integer $t$ there exists $\delta>0$
such that for every $\epsilon>0$ there exists
$\mathcal{G}^{1}_{n},\ldots,\mathcal{G}^{t}_{n}$ for all sufficiently large $n$
such that $\mathcal{G}^{i}_{n}$ is $K_{4}^{3}$-free and $|\mathcal{G}^{i}_{n}|>(5/9-\epsilon)\binom{n}{3}$ for all $i\in[t]$
and ${\rm ed}(\mathcal{G}^{i}_{n},\mathcal{G}^{j}_{n})>\delta n^3$ for all $\{i,j\}\subset [t]$.

Fix $t>0$ and let $0<\delta< \min\{1/t^6,10^{-8}\}$.
Fix $\epsilon>0$ and let $n$ be sufficiently large such that $$\binom{n}{3}-|\mathcal{G}(n,m)|>(5/9-\epsilon)\binom{n}{3}.$$
Let $m_{k} = \lceil 100\delta^{1/2}kn\rceil $ and $\mathcal{G}^{k}_{n} = \binom{[n]}{3}\setminus \mathcal{G}(n,m_k)$ for $k= 1,\ldots, t \le \lceil \delta^{-1/6} \rceil$.
Let $\{i,j\} \subset [t]$ and
assume $i<j$.
Notice that removing or adding one edge in a $n$-vertex $3$-graph $\mathcal{H}$
can  change the number of induced $K_{4}^{3-}$ by at most $n-3$.
Therefore, by Lemma~\ref{lemk43-} (and here we omit ceilings for ease of notation),
\begin{align}
    {\rm ed}(\mathcal{G}^{i}_{n},\mathcal{G}^{j}_{n})
    &= {\rm ed}(\mathcal{G}(n,m_i),\mathcal{G}(n,m_j))\notag\\
    &>\frac{\varphi(n,m_j)- \varphi(n,m_i)}{n}\notag\\
    &= \frac{(100\delta^{1/2}jn)^2(n-300\delta^{1/2}jn)(n-300\delta^{1/2}jn-3)}{6n}\notag\\
    &\quad -\frac{(100\delta^{1/2}in)^2(n-300\delta^{1/2}in)(n-300\delta^{1/2}in-3)}{6n}
    \notag\\
    &> 1000\delta n^3 \left(\left(1-\frac{3}{n}\right)(j^2-i^2) -300\delta^{1/2}\left(2-\frac{3}{n}\right) \left(j^3-i^3\right)+90000\delta\left(j^4-i^4\right)\right) \notag\\
    &> \delta n^3,\notag
\end{align}
where in the last inequality we used $\delta< \min\{1/t^6,10^{-8}\}$ and $i < j \le \lceil \delta^{-1/6} \rceil$.
This implies that  $\xi(K_{4}^{3})\ge t$ for all $t>0$. Therefore, $\xi(K_{4}^{3}) = \infty$.
\end{proof}

\section{Preliminaries}\label{SEC:prelim}
For a graph $G$ and two disjoint sets $A,B\subset V(G)$
denote by $G[A,B]$ the induced bipartite subgraph of $G$ with two parts $A$ and $B$.

Let $r \ge 2$ and $\mathcal{H}$ be an $r$-graph.
Recall that for every $v \in V(\mathcal{H})$ the link $L_{\mathcal{H}}(v)$ of $v$ in $\mathcal{H}$ is
\begin{align}
L_{\mathcal{H}}(v) = \left\{A\in \partial\mathcal{H}\colon A\cup \{v\}\in \mathcal{H}\right\}, \notag
\end{align}
the degree of $v$ in $\mathcal{H}$ is $d_{\mathcal{H}}(v) := |L_{\mathcal{H}}(v)|$,
and the minimum degree of $\mathcal{H}$ is $\delta(\mathcal{H}) := \min\{d_{\mathcal{H}}(v): v\in V(\mathcal{H})\}$.
For $S \subset V(\mathcal{H})$, the neighborhood\footnote{Note that this is not a standard definition for the neighborhood.
Some authors define the the neighborhood of an $s$-set $S$ to be its $(r-s)$-uniform link.}
of $S$ in $\mathcal{H}$ is
\[
N_{\mathcal{H}}(S) := \{v \in V(\mathcal{H})\setminus S: \exists E \in \mathcal{H} \text{ such that } \{v\}\cup S \subset E\}.
\]
Two vertices $u,v \in V(\mathcal{H})$ are adjacent in $\mathcal{H}$ if $u\in N_{\mathcal{H}}(v)$.
When it is clear from context we will omit the subscript $\mathcal{H}$ in the notations above.

Let $V(\mathcal{H}) = [n]$.
For $x = (x_1,\ldots,x_n)$ define the weight polynomial of a hypergraph $\mathcal{H}$ as
\[
p_{\mathcal{H}}(x) := \sum_{E\in \mathcal{H}} \prod_{i\in E} x_i.
\]
The standard $n$-simplex is
\[
\Delta^{n} := \left\{x \in \mathbb{R}^{n+1} : \sum_{i=1}^{n+1}x_i = 1 \text{ and } x_i \ge 0 \text{ for all }i \in [n+1] \right\}.
\]
The Lagrangian of $\mathcal{H}$ is
\[
\lambda(\mathcal{H}) := \max \left\{p_{\mathcal{H}}(x): x \in \Delta^{n-1} \right\}.
\]
Note that $\Delta^{n-1}$ is compact in $\mathbb{R}^{n}$ and $p_{\mathcal{H}}(x)$ is continuous,
so $\lambda(\mathcal{H})$ is well-defined.

Recall that in Section 1 we defined the blowup of an $r$-graph $\mathcal{T}$.
The next standard lemma gives a relationship between $\lambda(\mathcal{T})$ and the size of a blowup of $\mathcal{T}$.

\begin{lemma}\label{lemma-blowup-lagrang}
Let $r \ge 2$ and $\mathcal{T}$ and $\mathcal{H}$ be two $r$-graphs.
Suppose that $\mathcal{H}$ is a blowup of $\mathcal{T}$ with $v(\mathcal{H}) = n$.
Then $|\mathcal{H}| \le \lambda(\mathcal{T}) n^{r}$.
\end{lemma}
\begin{proof}
Suppose that $|V(\mathcal{T})| = s$ and $\mathcal{H} = \mathcal{T}(t)$ for some $t = (t_1, \ldots, t_s)$.
Then
\[
|\mathcal{H}| = \sum_{E \in \mathcal{T}} \prod_{i \in E}t_i = n^r \sum_{E \in \mathcal{T}} \prod_{i \in E}\frac{t_i}{n}
\le \lambda(\mathcal{T}) n^r,
\]
where the last inequality follows from the definition of $\lambda(\mathcal{T})$ and $\sum_{i\in[s]}t_i = n$.
\end{proof}

Given another $r$-graph $F$ we say $f : V(F) \to V(\mathcal{H})$ is a homomorphism if
$f(E) \in \mathcal{H}$ for all $E \in F$, i.e., $f$ preserves edges.
We say that $\mathcal{H}$ is $F$-hom-free if there is no homomorphism from $F$ to $\mathcal{H}$.
In other words, $\mathcal{H}$ is $F$-hom-free if and only if all blowups of $\mathcal{H}$ are $F$-free.
For a family $\mathcal{F}$ of $r$-graphs,
$\mathcal{H}$ is $\mathcal{F}$-hom-free if it is $F$-hom-free for all $F\in \mathcal{F}$.

An $r$-graph $F$ is $2$-covered if every $\{u,v\} \subset V(F)$ is contained in some $E \in F$,
and a family $\mathcal{F}$ is $2$-covered if all $F \in \mathcal{F}$ are $2$-covered.
It is easy to see that if $\mathcal{F}$ is $2$-covered,
then $\mathcal{H}$ is $\mathcal{F}$-free if and only if it is $\mathcal{F}$-hom-free.
Although $\mathcal{M}$ is not $2$-covered, we still have a similar result.

\begin{lemma}\label{lemma-M=homM}
A $3$-graph $\mathcal{H}$ is $\mathcal{M}$-free if and only if it is $\mathcal{M}$-hom-free.
\end{lemma}
\begin{proof}
The backward implication is clear. Now suppose conversely that $\mathcal{H}$ fails to be
$\mathcal{M}$-hom-free, i.e., that there is a homomorphism $f\colon V(F)\to V(\mathcal{H})$
for some $F\in\mathcal{M}$.
If $F\cong K_{5}^{3-}$, then $f$ is injective due to the fact that $K_{5}^{3-}$ is $2$-covered.
However, this implies that $K_{5}^{3-}\subset \mathcal{H}$, a contradiction.
Therefore, $F\in M_{2}\cup M_3$.
Clearly the restriction of $f$ to the core $S$
of $F$ is injective. So $f(F)\in\mathcal{K}^3_{|S|}\cap \mathcal{M}$ and
in view of $f(F)\subset \mathcal{H}$ it follows that $\mathcal{H}$ fails to be $\mathcal{M}$-free.
\end{proof}

\section{Tur\'{a}n number of $\mathcal{M}$}\label{SEC:proof-Turan-number-M}
In this section, we will prove Theorem \ref{thm-Turan-num-M}.
The first subsection contains some technical lemmas and calculations needed in the proof.

\subsection{Lagrangian of some $3$-graphs}
\begin{lemma}\label{lemma-langran-K43}
Suppose that $\mathcal{T}$ is a $3$-graph with at most four vertices.
Then $\lambda(\mathcal{T}) \le 1/16$.
\end{lemma}
\begin{proof}
Without loss of generality we may assume that $v(\mathcal{T}) = 4$ and $|\mathcal{T}| = 4$,
i.e., $\mathcal{T} \cong K_{4}^{3}$.
It is easy to see that
\[
p_{K_{4}^{3}}(x) = x_1x_2x_3+x_1x_2x_4+x_1x_3x_4+x_2x_3x_4 \le 4(1/4)^3 = 1/16.
\]
Therefore, $\lambda(\mathcal{T}) \le 1/16$.
\end{proof}

\begin{lemma}\label{lemma-lagrang-G26}
$\lambda(\mathcal{G}^2_{6}) \le 2/27$.
\end{lemma}
\begin{proof}
Notice that
\begin{align}
p_{\mathcal{G}^2_{6}}(x_1,\ldots,x_{6})
& = x_3x_6(x_1+x_2+x_4+x_5) \notag\\
& \quad + (x_1+x_2)(x_3+x_6)(x_4+x_5) + x_1x_2(x_4+x_5) + x_4x_5(x_1+x_2). \notag
\end{align}
Set $a=(x_3+x_6)/2$, $b = (x_1+x_2)/2$, $c=(x_4+x_5)/2$, $d=(b+c)/2$.
It follows from the AM-GM inequality that
\begin{align}
p_{\mathcal{G}^2_{6}}(x_1,\ldots,x_{6})
 \le 2a^2(b+c)+8abc+2bc(b+c)
& \le 4a^2d + 8ad^2 + 4d^3 \notag\\
& = 2\left((a+d)\cdot (a+d)\cdot 2d\right) \notag\\
& \le 2\left(\frac{(a+d)+ (a+d)+ 2d}{3}\right)^{3}
 = \frac{2}{27}. \notag
\end{align}
Moreover, equality holds only if $a=b=c=d=1/6$ and $x_1= \ldots =x_6 = 1/6$.
\end{proof}

\begin{lemma}\label{lemma-lagrang-K53=}
Suppose that $\mathcal{T}$ is a $3$-graph on five vertices with at most eight edges.
Then there exists an absolute constant $c>0$ such that $\lambda(\mathcal{T}) < 2/27 -c$.
\end{lemma}
\begin{proof}
This follows immediately from Lemma~\ref{lemma-lagrang-G26} and
the fact that every $3$-graph on $5$ vertices with $8$ edges is a subgraph of $\mathcal{G}_{6}^{2}$.
\end{proof}

\begin{lemma}\label{lemma-trans-star}
Let $\mathcal{T}$ be a $2$-covered $3$-graph on $k \ge 7$ vertices.
Suppose that $\tau(\mathcal{T}[S]) \le 1$ for all sets $S\subset V(\mathcal{T})$ with $|S| = 7$.
Then $\mathcal{T}$ is a star.
\end{lemma}

\noindent{\bf Remark.}
In fact, a weaker condition that $|S|=6$ is sufficient for the proof of Lemma~\ref{lemma-trans-star}.

\begin{proof}
Suppose that $\mathcal{T}$ is not a star.
Then for every vertex $v$ in $\mathcal{T}$ there exists an edge $E_v$ in $\mathcal{T}$ that does not contain $v$.

First notice that $\mathcal{T}$ cannot contain two disjoint edges. Therefore, $\mathcal{T}$ is intersecting.
Suppose that $\mathcal{T}$ contains two edges $E_1 = \{u,v_1,v_2\}$ and $E_2 = \{u, w_1, w_2\}$, where $\{v_1,v_2\}\cap \{w_1,w_2\} = \emptyset$.
Let $E_3\in \mathcal{T}$ be an edge that does not contain $u$.
Since $\mathcal{T}$ is intersecting, we may assume that $v_1,w_1\in E_3$.
Then, we have $|E_1\cup E_2\cup E_3|\le 6$, and $\tau(\{E_1,E_2,E_3\}) = 2$, a contradiction.
Therefore, we may assume that the intersection of every two edges in $\mathcal{T}$ has size two.
Let $E_1=\{u,v,w_1\}$ and $E_2 = \{u,v,w_2\}$ be two edges in $\mathcal{T}$.
By assumption there exists an edge $E_3 \in \mathcal{T}$ that does not contain $u$ and, hence, we have $E_3= \{v,w_1,w_2\}$.
Similarly there exists $E_4\in \mathcal{T}$ that does not contain $v$ and, hence, we have $E_4=\{u,w_1,w_2\}$.
Then, we have $|E_1\cup E_2\cup E_3\cup E_4|=4$, and $\tau(\{E_1,E_2,E_3,E_4\}) = 2$, a contradiction.
\end{proof}

\subsection{Proof of Theorem \ref{thm-Turan-num-M}}
In this section we complete the proof of Theorem \ref{thm-Turan-num-M}.

For $v \in V(\mathcal{H})$ and $E \in \mathcal{H}$,
$\mathcal{H}-v$ is obtained by removing $v$ and all edges containing $v$ from $\mathcal{H}$,
and $\mathcal{H}-E$ is obtained by removing $E$ from $\mathcal{H}$ and keeping $V(\mathcal{H})$ unchanged.

\begin{dfn}[Equivalence classes]\label{dfn-equivalent-class}
Let $\mathcal{H}$ be an $r$-graph and $u,v$ be two non-adjacent (i.e. no edge containing both) vertices in $\mathcal{H}$.
Then $u$ and $v$ are equivalent if $L(u) = L(v)$, otherwise they are \textit{non-equivalent}.
If $u$ and $v$ are equivalent, then we write $u\sim v$.
Let $C_v$ denote the equivalence class of $v$.
\end{dfn}

\noindent\textbf{Algorithm 1} (Symmetrization without cleaning)
Let $\mathcal{H}$ be an $r$-graph.
We perform the following operation as long as there are two non-adjacent non-equivalent vertices in $\mathcal{H}$.
Let $u,v$ be two such vertices with $d(u)\ge d(v)$.
Then we delete all vertices from $C_v$
and duplicate $u$ using $|C_v|$ vertices and still label these new vertices with labels in $C_v$.
Another way to view this operation is that we remove all edges in $\mathcal{H}$ that have nonempty intersection with $C_v$
and for every $E \in \mathcal{H}$ with $u \in E$ we add $E - \{u\} \cup \{v'\}$ for all $v' \in C_v$ into $\mathcal{H}$.
We terminate the process when there is no non-adjacent non-equivalent pair.

\bigskip

Note that the number of equivalence classes in $\mathcal{H}$ strictly decreases after each step that can be performed,
so Algorithm 1 always terminates.
On the other hand, since symmetrization only deletes and duplicates vertices, by Lemma \ref{lemma-M=homM},
Algorithm 1  preserves the $\mathcal{M}$-freeness of $\mathcal{H}$.
The following lemma is immediate from the definition.

\begin{lemma}\label{lemma-Ht-algo}
Let $\mathcal{H}_{t}$ be the $3$-graph obtained from $\mathcal{H}$ by applying Algorithm 1,
and let $T \subset V(\mathcal{H})$ such that $T$ contains exactly one vertex from each equivalence class of $\mathcal{H}_{t}$.
Then,
\begin{enumerate}[label=(\alph*)]
\item $|\mathcal{H}_{t}|\ge |\mathcal{H}|$.
\item $\mathcal{H}_{t}[T]$ is $2$-covered and $\mathcal{H}_{t}$ is a blowup of $\mathcal{H}_{t}[T]$.
\end{enumerate}
\end{lemma}

Now we are ready to finish the proof of Theorem \ref{thm-Turan-num-M}.
\begin{proof}[Proof of Theorem \ref{thm-Turan-num-M}]
Let $\mathcal{H}$ be an $\mathcal{M}$-free $3$-graph on $n$ vertices.
Apply Algorithm 1 to $\mathcal{H}$ and let $\mathcal{H}_{t}$ denote the resulting $3$-graph.
Let $T\subset V(\mathcal{H})$ such that $T$ contains exactly one vertex from each equivalent class in $\mathcal{H}_{t}$,
and let $\mathcal{T} = \mathcal{H}_{t}[T]$.
By Lemma 3.6, in order to prove $|\mathcal{H}| \le 2n^3/27$,
it suffices to show $|\mathcal{H}_{t}| \le 2n^3/27$.
Since $\mathcal{H}_{t}$ is a blowup of $\mathcal{T}$,
by Lemma 2.1, it suffices to show that $\lambda(\mathcal{T}) \le 2/27$.
Next, we will consider two cases depending on the size of $T$: either $|T| \ge 7$ or $|T| \le 6$.

\medskip

\noindent\textbf{Case 1}: $|T| \ge 7$.\\
Since $\mathcal{T}$ is $2$-covered and it is $M_2$-free, $\tau(\mathcal{T}[S]) \le 1$ for all $S \subset T$ with $|S| = 7$,
and it follows from Lemma \ref{lemma-trans-star} that $\mathcal{T}$ is a star.

Let us calculate $\lambda(\mathcal{T})$.
We may assume that $V(\mathcal{T}) = [s]$ for some integer $s$ and $1$ is the center of $\mathcal{T}$.
Then,
\[
\begin{split}
p_{\mathcal{T}}(x) & \le x_1 \left( \sum_{\{i,j\}\subset [s]\setminus\{1\}} x_ix_j \right) \le \frac{s-2}{2(s-1)} x_1 (1-x_1)^2
< \frac{1}{2} x_1 (1-x_1)^2 \le \frac{2}{27},
\end{split}
\]
which implies that $\lambda(\mathcal{T}) < 2/27$.

\medskip

\noindent\textbf{Case 2}: $|T| \le 6$.\\
If $|T| \le 5$, then Lemmas \ref{lemma-langran-K43} and \ref{lemma-lagrang-K53=} imply that $\lambda(\mathcal{T}) < 0.67277$.
So we may assume that $|T| = 6$.

Lemma \ref{lemma-Ht-algo} implies that $\mathcal{T}$ is $2$-covered, so $\mathcal{T}\in \mathcal{K}_{6}^3$.
Since $\mathcal{H}_{t}$ does not contain any member in $M_3$ as a subgraph,
either $\mathcal{T} \subset \mathcal{G}^1_{n}$ or $\mathcal{T} \subset \mathcal{G}^2_{n}$ for some $n \ge 6$.
Due to the fact that $\mathcal{T}$ is $2$-covered again, either $\mathcal{T}$ is a star or $\mathcal{T} \subset \mathcal{G}^2_{6}$.
The former case has been handled by Case 1, so we may assume that $\mathcal{T} \subset \mathcal{G}^2_{6}$,
and it follows from Lemma \ref{lemma-lagrang-G26} that $\lambda(\mathcal{T}) \le \lambda(\mathcal{G}^2_{6}) \le 2/27$.
\end{proof}

\section{Stability of $\mathcal{M}$}\label{SEC:proof-stability-M}
In this section we will prove Theorem \ref{thm-2-stable}.
First we present an algorithm and some lemmas that will be used in the proof.

\subsection{Symmetrization}
Let $0 \le \alpha \le 1$ and $\mathcal{H}$ be a $3$-graph.
Then $\mathcal{H}$ is $\alpha$-dense if $\delta(\mathcal{H}) \ge \alpha\binom{v(\mathcal{H})-1}{2}$.
Let $(V,\prec_{V})$ be a poset on $V$ with relation $\prec_{V}$.
For $S \subset V$ the induced poset of $(V,\prec_{V})$ on $S$ is denoted by $(S,\prec_{V})$.

\medskip

\noindent\textbf{Algorithm 2} ({Symmetrization and cleaning with threshold $\alpha$}).\\
\noindent\textbf{Input}: A $3$-graph $\mathcal{H}$.\\
\noindent\textbf{Operation}:
\begin{itemize}
\item
\textbf{Initial step}:
If $\delta(\mathcal{H}) \ge \alpha \binom{v(\mathcal{H})-1}{2}$,
then let $\mathcal{H}_{0} = \mathcal{H}$ and $V_0 = V(\mathcal{H})$.
Otherwise, we keep deleting vertices with the minimum degree one by one until
the remaining $3$-graph $\mathcal{H}_{0}$ is either empty or $\delta(\mathcal{H}_0) \ge \alpha \binom{v(\mathcal{H}_0)-1}{2}$.
Let $Z_0$ be the set of deleted vertices during this process so that $V_0:= V(\mathcal{H}_0) = V(\mathcal{H})-Z_0$.
\end{itemize}

Let $(V_0,\prec_{V_0})$ be the poset with $V_0$ itself an antichain,
i.e., there is no relation between any two vertices in $V_0$.

Suppose we are at the $i$-th step for some $i \ge 1$.
We terminate the algorithm if either
\begin{itemize}
\item[(a)]
$\mathcal{H}_{i-1} = \emptyset$ or

\item[(b)]
$\delta(\mathcal{H}_{i-1}) \ge \alpha\binom{v(\mathcal{H}_{i-1})-1}{2}$
and there is no pair of non-adjacent non-equivalent vertices.
\end{itemize}
Otherwise, we iterate the following two operations.

\begin{itemize}
\item[\empty]
\textbf{Step 1 (Symmetrization)}:
If $\mathcal{H}_{i-1}$ contains no pair of non-adjacent non-equivalent vertices,
then let $\mathcal{G}_{i} = \mathcal{H}_{i-1}$ and go to Step 2.
Otherwise, choose two non-adjacent non-equivalent vertices $u,v \in V(\mathcal{H}_{i-1})$ and assume that $d(u) \ge d(v)$.
Delete all vertices in $C_v$ and add $|C_v|$ new vertices into $C_u$ by duplicating $u$ and label these new vertices with labels in $C_v$,
which is the same as what we did in Algorithm 1.
Let $\mathcal{G}_{i}$ denote the resulting $r$-graph,
and update the poset $(V_{i-1},\prec_{V_{i-1}})$ by adding the following relations: $v' \prec u'$ for all $v' \in C_v$ and all $u' \in C_u$.
This new poset is well-defined as one will see from the following operations that
once two equivalence classes are merged they will never be split.

\item[\empty]
\textbf{Step 2 (Cleaning)}:
If $\delta(\mathcal{G}_{i}) \ge \alpha\binom{v(\mathcal{G}_{i})-1}{2}$, then let $\mathcal{H}_{i} = \mathcal{G}_{i}$
and $(V_i,\prec_{V_{i}}) = (V_{i-1},\prec_{V_{i-1}})$.
Otherwise let $\mathcal{L} = \mathcal{G}_{i}$ and repeat Steps 2.1 and 2.2.

\item[\empty]
\textbf{Step 2.1}: Let $B = \left\{a\in V(\mathcal{L}): d_{\mathcal{L}}(a) = \delta(\mathcal{L}) \right\}$
and choose a minimal element $z \in (B,\prec_{V_{i-1}})$.
Replace $\mathcal{L}$, $V_{i-1}$, and $(V_{i-1},\prec_{V_{i-1}})$
by $\mathcal{L} - z$, $V_{i-1}\setminus \{z\}$, and $(V_{i-1}\setminus\{z\},\prec_{V_{i-1}})$, respectively.
\item[\empty]
\textbf{Step 2.2}: If $\delta(\mathcal{L}) \ge \alpha \binom{v(\mathcal{L})-1}{2}$ or $\mathcal{L}=\emptyset$, then stop. Otherwise, go to Step 2.1.

Let $\mathcal{H}_{i} = \mathcal{L}$ and $(V_i,\prec_{V_i}) = (V_{i-1},\prec_{V_{i-1}})$.
Let $Z_{i}$ denote the set of vertices removed by Step 2.1 so that $\mathcal{H}_{i} = \mathcal{G}_{i} - Z_{i}$.
\end{itemize}

\noindent\textbf{Output}:  A $3$-graph $\mathcal{H}_{t}$ for some $t$ such that
either $\mathcal{H}_{t}$ is empty or
$\delta(\mathcal{H}_{t}) \ge \alpha\binom{v(\mathcal{H}_{t})-1}{2}$
and there is no pair of non-adjacent non-equivalent vertices in $\mathcal{H}_{t}$.

\medskip

\noindent{\bf Remark.}
The point of Step 2 is that the symmetrization step (Step 1)
could potentially bring down the degree of some of the vertices in the hypergraph,
making the pruning step (Step 2) necessary.

\medskip

Let $\epsilon > 0$ be sufficiently small and $n$ be sufficiently large.
Let $\mathcal{H}$ be an $\mathcal{M}$-free $3$-graph on $n$ vertices with $|\mathcal{H}| \ge 2n^3/27 - \epsilon n^3$.
Apply Algorithm 2 to $\mathcal{H}$ with threshold $\alpha = 4/9 - 3 \epsilon^{1/2}$ and suppose that it stops at the $t$-th step.
Let $\mathcal{H}_{t}$ denote the resulting $3$-graph and $W = V(\mathcal{H}_{t})$ and $\tilde{n} = |W|$.
For $0 \le i \le t$ let $\widetilde{\mathcal{H}}_{i}= \mathcal{H}_i[W]$ and $\widetilde{\mathcal{G}}_{i}= \mathcal{G}_i[W]$.
Note that $\widetilde{\mathcal{H}}_{0} = \mathcal{H}[W]$ and $\widetilde{\mathcal{G}}_{0}= \mathcal{G}[W]$,
and we will omit the subscript $0$ if there is no cause for confusion.
Let $Z = \bigcup_{i=0}^{t}Z_i$ be the set of vertices in $\mathcal{H}$ that were removed by Algorithm 2.
In the rest of the proof we will focus on $\widetilde{\mathcal{H}}_{i}$ and $\widetilde{\mathcal{G}}_{i}$.
Notice from Algorithm 2 that $\mathcal{H}_{i} = \mathcal{G}_{i} - Z_{i}$ and $Z_i \subset V(\mathcal{H}) \setminus W$,
therefore, $\widetilde{\mathcal{H}}_{i} = \widetilde{\mathcal{G}}_{i}$ for all $1 \le i \le t$.

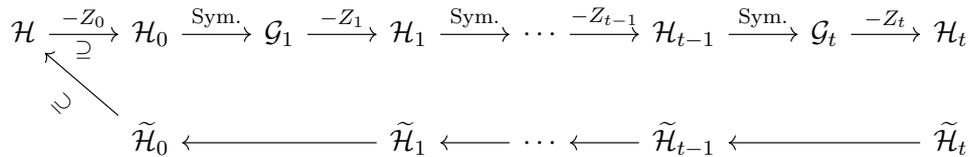
\begin{figure}[htbp]
\centering
\begin{tikzcd}
  \mathcal{H} \arrow{r}[swap]{\supseteq}{-Z_0}
& \mathcal{H}_{0} \arrow{r}{\text{Sym.}}
& \mathcal{G}_{1} \arrow{r}{-Z_1}
& \mathcal{H}_{1} \arrow{r}{\text{Sym.}}
& \cdots \arrow{r}{-Z_{t-1}}
& \mathcal{H}_{t-1} \arrow{r}{\text{Sym.}}
& \mathcal{G}_{t} \arrow{r}{-Z_t}
& \mathcal{H}_{t} \\
\empty
& \widetilde{\mathcal{H}}_{0} \arrow{ul}{\rotatebox{-45}{$\supseteq$}}
& & \widetilde{\mathcal{H}}_{1} \arrow{ll}
& \cdots \arrow{l}
& \widetilde{\mathcal{H}}_{t-1} \arrow{l}
& & \widetilde{\mathcal{H}}_{t} \arrow{ll}
\end{tikzcd}
\caption{The first line contains the $3$-graphs produced by Algorithm 2 and the second line contains the corresponding induced $3$-graphs on $W$.}
\label{figure:Algorithm2}
\end{figure}

\begin{lemma}\label{lemma-Hi+1-and-Hi}
For every $i \in [t]$ either $\widetilde{\mathcal{H}}_{i-1} = \widetilde{\mathcal{H}}_{i}$ or
there exist two nonempty equivalence classes $V_i \subset W$ and $U_i \subset W$
in $\widetilde{\mathcal{H}}_{i-1}$ such that
$\widetilde{\mathcal{H}}_{i}$ is obtained from $\widetilde{\mathcal{H}}_{i-1}$ by deleting all vertices in $V_i$
and adding $|V_i|$ new vertices by duplicating some vertex in $U_i$.
\end{lemma}
\begin{proof}
Fix $1 \le i \le t$ and suppose that in forming $\mathcal{G}_i$ from $\mathcal{H}_{i-1}$ in Algorithm 2 we deleted all vertices in $C_v$
and added $|C_v|$ new vertices by duplicating some $u \in C_u$,
where $C_v$ (resp. $C_u$) is the equivalence class of $v \in V(\mathcal{H}_{i-1})$ (resp. $u \in V(\mathcal{H}_{i-1})$) in $\mathcal{H}_{i-1}$.
Let $\widehat{C}_{u} = C_v \cup C_u$ and
notice that for every $i \le j \le t$
the set $\widehat{C}_{u} \cap V(\mathcal{G}_{j})$
(resp. $\widehat{C}_{u} \cap V(\mathcal{H}_{j})$) is an equivalence class in $\mathcal{G}_j$ (resp. $\mathcal{H}_j$).

If $C_v \cap W = \emptyset$, then $\widetilde{\mathcal{H}}_{i-1} = \widetilde{\mathcal{H}}_{i}$ and we are done.
So we may assume that $C_v \cap W \neq \emptyset$.

First, we claim that $C_u \subset W$.
Indeed, suppose that there exists $u' \in C_u \setminus W$.
Then it means that $u'$ was removed at the $j$-th step for some $i \le j \le t$.
Since all $v' \in C_{v}$ satisfy $v' \prec_{V_k} u'$ and $d_{\mathcal{G}_k}(v') = d_{\mathcal{G}_k}(u')$ for all $i \le k \le j$,
by definition of Algorithm 2 all vertices in $C_v$ must be removed before $u'$ was removed,
which implies that $C_v \cap W = \emptyset$, a contradiction.
Therefore, $C_u \subset W$.

Let $V_i = C_v \cap W$ and $U_i = C_u$ and note that neither of them is empty.
Since $C_v$ and $C_u$ are equivalence classes in $\mathcal{H}_{i-1}$, $V_i$ and $U_i$ are equivalence classes in $\widetilde{\mathcal{H}}_{i-1}$.
According to Algorithm 2, $\widetilde{\mathcal{H}}_{i}$ is obtained from $\widetilde{\mathcal{H}}_{i-1}$ by deleting all vertices in $V_i$
and adding $|V_i|$ new vertices by duplicating some vertex in $U_i$.
\end{proof}

The following two lemmas show that the size of the set $Z$ of vertices removed by Algorithm 2 is small,
and the induced subgraph $\widetilde{\mathcal{H}}_i$ of $\mathcal{H}_{i}$ on $W$ has a large minimum degree for $0\le i \le t$.
Their proofs can be found in \cite{BIJ17}.

\begin{lemma}\label{lemma-size-Z-W}
We have $|Z| \le 3 \epsilon^{1/2} n$, and hence $\tilde{n} \ge n - 3 \epsilon^{1/2} n$.
\end{lemma}

\begin{lemma}\label{lemma-low-bound-tilde-Hi}
For all $0 \le i \le t$,
\[
\delta(\widetilde{\mathcal{H}}_{i}) > \left({4}/{9}- 10 \epsilon^{1/2} \right)\binom{\tilde{n}-1}{2}
\]
and, in particular,
\[
|\widetilde{\mathcal{H}}_{i}| > \left({4}/{9}- 10 \epsilon^{1/2} \right)\binom{\tilde{n}}{3}.
\]
\end{lemma}

Notice that $\widetilde{\mathcal{H}}_t = \mathcal{H}_t$ and $\widetilde{\mathcal{H}}_0 = \mathcal{H}[W]$.
In order to prove Theorem \ref{thm-2-stable} it suffices to
show that $\widetilde{\mathcal{H}}_0$ is either semibipartite or $\mathcal{G}^2_6$-colorable.
We will proceed by backward induction on $i$. The following lemma establishes the base case of the induction.

\begin{lemma}\label{lemma-Ht-semi-or-G2-color}
Let $T \subset W$ such that $T$ contains exactly one vertex in each equivalence class in $\mathcal{H}_{t}$
and $\mathcal{T} = {\mathcal{H}}_{t}[T]$.
Then, either $\mathcal{T}$ is a star or $\mathcal{T} \subset \mathcal{G}_{6}^{2}$ and, in particular,
$\mathcal{H}_{t}$ is either semibipartite or $\mathcal{G}^2_6$-colorable.
\end{lemma}
\begin{proof}
First we claim that $|T| \ge 6$.
Indeed, suppose that $|T| \le 5$.
Then, Lemmas \ref{lemma-langran-K43} and \ref{lemma-lagrang-K53=} imply that
$\lambda(\mathcal{T}) < 0.067277$.
It follows from Lemma \ref{lemma-blowup-lagrang} that
$|\mathcal{H}_{t}| < 0.067277 \tilde{n}^3 < (4/9- 10 \epsilon^{1/2})\binom{\tilde{n}}{3}$, which contradicts Lemma \ref{lemma-low-bound-tilde-Hi}.
Therefore, $|T| \ge 6$.

Suppose that $|T| \ge 7$.
Since $\mathcal{T}$ is $2$-covered and $M_2$-free, $\tau(\mathcal{T}[S]) \le 1$ for all $S \subset T$ with $|S| = 7$.
So by Lemma \ref{lemma-trans-star}, $\mathcal{T}$ is a star, and hence $\mathcal{H}_{t}$ is semibipartite.

Suppose that $|T| = 6$.
Since $\mathcal{T}\in \mathcal{K}_{6}^3$ and $\mathcal{H}$ is $M_3$-free,
either $\mathcal{T} \subset \mathcal{G}^1_{m}$ or $\mathcal{T} \subset \mathcal{G}^2_{m}$ for some integer $m \ge 6$.
Moreover, due to the fact that $\mathcal{T}$ is $2$-covered,
either $\mathcal{T}$ is a star or $\mathcal{T} \subset \mathcal{G}^2_{6}$.
In the former case, $\mathcal{H}_{t}$ is semibipartite,
and in the latter case, $\mathcal{H}_{t}$ is $\mathcal{G}^2_6$-colorable.
\end{proof}

Next, we will consider two cases in the following two subsections depending on the structure of $\mathcal{H}_{t}$.


\subsection{Semibipartite}
In this section we will prove the following statement.

\begin{lemma}\label{main-lemma-Ht-semibipartite}
Suppose that $\mathcal{H}_{t}$ is semibipartite.
Then $\widetilde{\mathcal{H}}_{i}$ is semibipartite for all $0\le i \le t$.
In particular, $\mathcal{H}[W] = \widetilde{\mathcal{H}}_{0}$ is semibipartite.
\end{lemma}

We will use the following stability theorem due to F\"{u}redi, Pikhurko, and Simonovits \cite{FPS05} to prove Lemma \ref{main-lemma-Ht-semibipartite}.

Let $\mathbb{F}_{3,2}$ be the $3$-graph with vertex set $[5]$ and edges set $\{123,124,125,345\}$.
F\"{u}redi, Pikhurko, and Simonovits \cite{FPS05} proved that if $n$ is sufficiently large,
then $\mathcal{G}_{n}^{1}$ is the unique $\mathbb{F}_{3,2}$-free $3$-graph on $n$ vertices with the maximum number of edges.
Moreover, they proved the following strong stability result.

\begin{thm}[F\"{u}redi-Pikhurko-Simonovits \cite{FPS05}]\label{THM:FPS-independent-neighbor}
Let $\gamma \le 1/125$ be fixed and $n\ge n_0$.
Let $\mathcal{H}$ be an $\mathbb{F}_{3,2}$-free $3$-graph on $n$ vertices with
$\delta(\mathcal{H})> (4/9-\gamma)\binom{n}{2}$.
Then $\mathcal{H}$ is semibipartite.
\end{thm}

Now we prove Lemma \ref{main-lemma-Ht-semibipartite}.
\begin{proof}[Proof of Lemma \ref{main-lemma-Ht-semibipartite}]
The proof is by backward induction on $i$ and the base case is $i = t$ as $\widetilde{\mathcal{H}}_{t} = \mathcal{H}_{t}$.
Now suppose that $\widetilde{\mathcal{H}}_{i+1}$ is semibipartite with two parts $A^{i+1}$ and $B^{i+1}$ for some $0 \le i \le t-1$,
and every edge in $\widetilde{\mathcal{H}}_{i+1}$ has exactly one vertex in $A^{i+1}$.
We may assume that both $A^{i+1}$ and $B^{i+1}$ are union of some equivalence classes.
Our goal is to show that $\widetilde{\mathcal{H}}_{i}$ is also semibipartite.

Recall that $\epsilon > 0$ is a sufficiently small constant and $\tilde{n}$ is a sufficiently large integer.

Denote by $\widehat{\mathcal{G}}$ the semibipartite $3$-graph on $W$
that consists of all triples that have exactly one vertex in $A^{i+1}$.
Notice that $\widetilde{\mathcal{H}}_{i+1} \subset \widehat{\mathcal{G}}$ and
$L_{\widetilde{\mathcal{H}}_{i+1}}(w) \subset L_{\widehat{\mathcal{G}}}(w)$ for all $w\in W$.

\begin{claim}\label{claim-size-Ai+1-Bi+1}
We have $\left| |A^{i+1}| - {\tilde{n}}/{3} \right| < 4 \epsilon^{1/4} \tilde{n}$ and
$\left||B^{i+1}| - {2\tilde{n}}/{3} \right| < 4 \epsilon^{1/4} \tilde{n}$.
\end{claim}
\begin{proof}[Proof of Claim \ref{claim-size-Ai+1-Bi+1}]
Let $\beta = |B^{i+1}|$.
Since $\widetilde{\mathcal{H}}_{i+1}$ is semibipartite,
\[
|\widetilde{\mathcal{H}}_{i+1}| \le \left( \tilde{n} - \beta \right)\binom{\beta}{2}.
\]
On the other hand, by Lemma \ref{lemma-low-bound-tilde-Hi},
$|\widetilde{\mathcal{H}}_{i+1}| \ge (4/9- 10 \epsilon^{1/2})\binom{\tilde{n}}{3}$.
Therefore,
\[
(4/9- 10 \epsilon^{1/2})\binom{\tilde{n}}{3} \le  \left( \tilde{n} - \beta \right)\binom{\beta}{2},
\]
which implies that $(2/3 - 4 \epsilon^{1/4}) \tilde{n}< \beta < (2/3 + 4 \epsilon^{1/4}) \tilde{n}$.
\end{proof}

For every vertex $w\in W$ let $M_{w} = L_{\widehat{\mathcal{G}}}(w) \setminus L_{\widetilde{\mathcal{H}}_{i+1}}(w)$.
Members in $M_w$ are called missing edges of $L_{\widetilde{\mathcal{H}}_{i+1}}(w)$.

\begin{claim}\label{CLAIM:number-of-missing-edges-semibipartite}
We have $|M_w| \le 10 \epsilon^{1/4}\tilde{n}^{2}$ for all $w\in W$.
\end{claim}
\begin{proof}[Proof of Claim~\ref{CLAIM:number-of-missing-edges-semibipartite}]
If $w\in A^{i+1}$, then $L_{\widehat{\mathcal{G}}}(w)$ is a complete graph on $B^{i+1}$.
If $w\in B^{i+1}$, then $L_{\widehat{\mathcal{G}}}(w)$ is a complete bipartite graph with two parts $A^{i+1}$ and $B^{i+1}$.
Claim~\ref{claim-size-Ai+1-Bi+1} and Lemma~\ref{lemma-low-bound-tilde-Hi} imply that
for every $w\in A^{i+1}$ we have
\begin{align}
|M_w|
\le \binom{2\tilde{n}/3+4\epsilon^{1/4}\tilde{n}}{2} - (4/9-10\epsilon^{1/2})\binom{\tilde{n}}{2}
\le 10 \epsilon^{1/4}\tilde{n}^{2},  \notag
\end{align}
and for every $w\in B^{i+1}$ we have
\begin{align}
|M_w|
\le (\tilde{n}/3+4\epsilon^{1/4}\tilde{n})(2\tilde{n}/3+4\epsilon^{1/4}\tilde{n}) - (4/9-10\epsilon^{1/2})\binom{\tilde{n}}{2}
\le 10 \epsilon^{1/4}\tilde{n}^{2}.  \notag
\end{align}
\end{proof}

\begin{figure}[htbp]
\centering
\begin{tikzpicture}[xscale=4,yscale=4]
    \node (a) at (-0.4,0.2) {};
    \node (b) at (0.4,0.1) {};
    \node (b') at (0.4,0.3) {};
    \node (v) at (0,-0.6) {};
    \node (c1) at (0,0) {};
    \node (c2) at (0,-0.1) {};
    \node (c3) at (0,-0.2) {};
    \fill (a) circle (0.015) node [right] {$a$};
    \fill (b) circle (0.015) node [left] {$b$};
    \fill (b') circle (0.015) node [left] {$b'$};
    \fill (v) circle (0.015) node [above] {$v$};
    \fill (c1) circle (0.012);
    \fill (c2) circle (0.012);
    \fill (c3) circle (0.012);
    \draw[rotate=0,line width=0.8pt] (-0.4,0) ellipse [x radius=0.15, y radius=0.3];
    \draw[rotate=0,line width=0.8pt] (0.4,0) ellipse [x radius=0.2, y radius=0.6];
    \draw[line width=0.8pt] (0,-0.6) circle [radius=0.1];
    \draw[line width=0.8pt,color=sqsqsq,fill=sqsqsq,fill opacity=0.15]
        (-0.4-0.05,0.2) to [out = 90, in = 180] (-0.4,0.2+0.05) to [out = 0, in = 180]
        (0.4,0.3+0.05) to [out = 0, in = 0] (0.4,0.1-0.05) to [out = 180, in = 0] (-0.4,0.2-0.05)
        to [out = 180, in = 270] (-0.4-0.05,0.2);
    \draw[line width=0.8pt,color=sqsqsq,fill=sqsqsq,fill opacity=0.15]
        (0,-0.6-0.05) to [out = 0, in = 270] (0+0.05, -0.6)
        to [out = 90, in = 180] (0.4, -0.5-0.05)
        to [out = 0, in = 270] (0.4+0.05,-0.5) to [out = 90, in = 0] (0.4, -0.5+0.05)
        to [out = 180, in = 0] (-0.4, -0.2 +0.05) to [out = 180, in = 90]
        (-0.4-0.05,-0.2) to [out = 270, in = 180] (-0.4, -0.2-0.05) to [out = 0, in = 90] (0-0.05, -0.6)
        to [out = 270, in = 180] (0,-0.6-0.05);
    \node at (-0.4,-0.3-0.07) {$L$};
    \node at (0.4,-0.6-0.07) {$R$};
    \node at (0,-0.7-0.07) {$C_v$};
\end{tikzpicture}
\caption{The $3$-graph $\widetilde{\mathcal{H}}_{i+1}$ is obtained from $\widetilde{\mathcal{H}}_{i}$
by symmetrizing $C_v$ to some equivalence class $C_u$ that is contained in $L$ or $R$.}
\label{figure:Hi-semibipartite}
\end{figure}
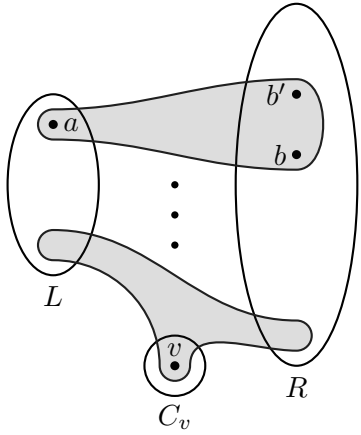

Lemma \ref{lemma-Hi+1-and-Hi} implies that either $\widetilde{\mathcal{H}}_{i} = \widetilde{\mathcal{H}}_{i+1}$
or there exists two equivalence classes $C_v$ and $C_u$ in $\widetilde{\mathcal{H}}_{i}$ such that $\widetilde{\mathcal{H}}_{i+1}$
is obtained from $\widetilde{\mathcal{H}}_{i}$ by symmetrizing $C_v$ to $C_u$ (see Figure~\ref{figure:Hi-semibipartite}).
In the former case, there is nothing to prove, so we may assume that we are in the latter case.

Let $L = A^{i+1} \setminus C_v$, $R = B^{i+1} \setminus C_v$ and $W' = W \setminus C_v$.
Since $\widetilde{\mathcal{H}}_{i+1}$ is obtained from $\widetilde{\mathcal{H}}_{i}$ by symmetrizing $C_v$ to
$C_u$ that is contained in either $L$ or $R$,
it follows that
\begin{align}\label{equ:Cv-semipartite}
{\rm either}\quad \left(L\cup C_v, R\right) = \left(A^{i+1}, B^{i+1}\right)
\quad{\rm or}\quad
\left(L, R\cup C_v\right) = \left(A^{i+1}, B^{i+1}\right),
\end{align}
and in particular, $L \neq \emptyset$ and $R \neq \emptyset$.

Since $C_v$ is an equivalence class in $\widetilde{\mathcal{H}}_{i}$,
$L_{\widetilde{\mathcal{H}}_{i}}(v') = L_{\widetilde{\mathcal{H}}_{i}}(v)$ for all $v' \in C_v$.
Thus we may just focus on $v$.
Notice that in forming $\widetilde{\mathcal{H}}_{i+1}$ from $\widetilde{\mathcal{H}}_{i}$ we only delete and add edges that
have nonempty intersection with $C_v$,
so $\widetilde{\mathcal{H}}_{i}[W'] = \widetilde{\mathcal{H}}_{i+1}[W']$.
Since $\widetilde{\mathcal{H}}_{i+1}$ is semibipartite,
it follows that $\widetilde{\mathcal{H}}_{i}[W']$ is semibipartite with two parts $L$ and $R$.

\begin{claim}\label{claim-low-bound-N(v)-cap-R}
We have $|N_{\widetilde{\mathcal{H}}_{i}}(v) \cap R| \ge \left(  1/3 - 5 \epsilon^{1/4}\right) \tilde{n}$.
In particular, $|R| \ge \left(  1/3 - 5 \epsilon^{1/4}\right) \tilde{n}$.
\end{claim}
\begin{proof}[Proof of Claim \ref{claim-low-bound-N(v)-cap-R}]
By Lemma \ref{lemma-low-bound-tilde-Hi},
\[
\binom{| N_{\widetilde{\mathcal{H}}_{i}}(v) |}{2}
\ge d_{\widetilde{\mathcal{H}}_{i}}(v)
\ge \left( {4}/{9} - 10 \epsilon^{1/2} \right)\binom{\tilde{n}-1}{2},
\]
which implies that $| N_{\widetilde{\mathcal{H}}_{i}}(v)| \ge (2/3 - 15\epsilon^{1/2})\tilde{n}$.
By Claim \ref{claim-size-Ai+1-Bi+1}, $|L| \le (1/3 + 4 \epsilon^{1/4})\tilde{n}$,
and hence
\[
| N_{\widetilde{\mathcal{H}}_{i}}(v) \cap R|
\ge  \left( {2}/{3} - 15 \epsilon^{1/2} \right) \tilde{n} - \left({1}/{3} + 4 \epsilon^{1/4}\right)\tilde{n}
> \left({1}/{3} - 5 \epsilon^{1/4}\right)\tilde{n}.
\]
\end{proof}

\begin{claim}\label{CLAIM:w-has-many-neighbors-in-R}
For every vertex $w\in W'$ we have $|N_{\widetilde{\mathcal{H}}_{i}} \cap R| \ge |R| - \tilde{n}/100$.
\end{claim}
\begin{proof}[Proof of Claim~\ref{CLAIM:w-has-many-neighbors-in-R}]
Notice that $L_{\widetilde{\mathcal{H}}_{i+1}}(w)[W'] = L_{\widetilde{\mathcal{H}}_{i}}(w)[W']$ for every $w\in W'$.
Therefore, for every $w\in L$ we have
$|L_{\widehat{\mathcal{G}}}(w)[R] \setminus L_{\widetilde{\mathcal{H}}_{i}}(w)[R]|
\le |M_{w}| \le 10\epsilon^{1/4}\tilde{n}^2$.
By Claim~\ref{claim-low-bound-N(v)-cap-R}, $|R| \ge \left(1/3 - 5 \epsilon^{1/4}\right) \tilde{n}$.
So the number of vertices in $R$ that have degree $0$ in $L_{\widetilde{\mathcal{H}}_{i}}(w)[R]$ is at most
$2\times 10\epsilon^{1/4}\tilde{n}^2/|R| < 80 \epsilon^{1/4} \tilde{n} < \tilde{n}/100$.

Now fix $u\in R$.
If $|L| \ge \tilde{n}/100$, then a
similar argument as above applied to graphs $L_{\widehat{\mathcal{G}}}(u)[L,R]$ and $L_{\widetilde{\mathcal{H}}_{i}}(u)[L,R]$
yields the number of vertices in $R$ that have degree $0$ in $L_{\widetilde{\mathcal{H}}_{i}}(u)[L,R]$ is at most
$2\times 10\epsilon^{1/4}\tilde{n}^2/|L| \le 2000 \epsilon^{1/4} \tilde{n} < \tilde{n}/100$.

So we may assume that $|L| < \tilde{n}/100$.
Due to Claim~\ref{claim-size-Ai+1-Bi+1} and $(\ref{equ:Cv-semipartite})$,
we must have $C_v \cup L = A^{i+1}$
since otherwise we would have
$|L| = |A^{i+1}| \ge (1/6-4\epsilon^{1/4})\tilde{n} > \tilde{n}/100$, a contradiction.
In particular, $|C_v| \le |A^{i+1}| \le (1/3+4\epsilon^{1/4})\tilde{n}$ and $|R| = |B^{i+1}|$.
Notice that $L_{\widetilde{\mathcal{H}}_{i}}(u)$ is a $3$-partite graph with three parts $L$, $R$, and $C_v$
(note that $C_v$ is an equivalence class, so no pair in $C_v$ is covered).
Let $x$ denote the number of vertices in $R$ that have degree $0$ in $L_{\widetilde{\mathcal{H}}_{i}}(u)$,
and note that for a vertex $u'\in R$ with degree $0$ in $L_{\widetilde{\mathcal{H}}_{i}}(u)$
every vertex $u''\in L\cup C_v$ forms a pair $\{u',u''\}$ that is not contained in $L_{\widetilde{\mathcal{H}}_{i}}(u)$.
Then due to $d_{\widetilde{\mathcal{H}}_{i}}(u) \ge (4/9-10\epsilon^{1/2})\binom{\tilde{n}-1}{2}$,
we have
\begin{align}
(4/9-10\epsilon^{1/2})\binom{\tilde{n}-1}{2} +  x(|L|+|C_v|)
\le d_{\widetilde{\mathcal{H}}_{i}}(u) + x(|L|+|C_v|)
\le |L||C_v|+|R|(|L|+|C_v|), \notag
\end{align}
which implies that $x\le \tilde{n}/100$.
\end{proof}

We may assume that $\widetilde{\mathcal{H}}_i$ contains a copy of $\mathbb{F}_{3,2}$,
since otherwise by Theorem~\ref{THM:FPS-independent-neighbor} we are done.
Let $S \subset W$ be a set of size $5$ such that $\mathbb{F}_{3,2} \subset \widetilde{\mathcal{H}}_i[S]$.
Observe that $S \cap C_v \neq \emptyset$ and due to the fact that $\mathbb{F}_{3,2}$ is $2$-covered,
we actually have $|S\cap C_v| = 1$.
We may assume that $\{v\} = S\cap C_v$.
Let $\{w_1,w_2,w_3,w_4\} = S\setminus \{v\}$.
Define $R' = R\cap N_{\widetilde{\mathcal{H}}_i}(v)\cap \left(\bigcap_{j\in[4]}N_{\widetilde{\mathcal{H}}_i}(w_j)\right)$.
Then Claims~\ref{claim-low-bound-N(v)-cap-R} and~\ref{CLAIM:w-has-many-neighbors-in-R}
imply that
$|R'| \ge \left(  1/3 - 5 \epsilon^{1/4}\right) \tilde{n} - 4 \times \tilde{n}/100 > \tilde{n}/6$.
Fix a vertex $u\in L$ (it is possible that $u\in \{w_1,w_2,w_3,w_4\}$).
By Claim~\ref{CLAIM:number-of-missing-edges-semibipartite},
$|L_{\widehat{\mathcal{G}}}(u)[R'] \setminus L_{\widetilde{\mathcal{H}}_{i}}(u)[R']|
\le |M_{u}| \le 10\epsilon^{1/4}\tilde{n}^2$.
So there exists an edge $w_5w_6 \in L_{\widetilde{\mathcal{H}}_{i}}(u)[R']$.
Let $E \subset \widetilde{\mathcal{H}}_{i}$ be a set of edges of size at most $10$
that covers all pairs in $\{v,w_1,w_2,w_3,w_4\} \times \{w_5,w_6\}$, and let
$F = \widetilde{\mathcal{H}}_{i}[\{v,w_1,w_2,w_3,w_4\}] \cup \{uw_5w_6\} \cup E$.
Then it is easy to see that $F$ is a member in $M_2$
(since $\mathbb{F}_{3,2}\subset \widetilde{\mathcal{H}}_{i}[\{v,w_1,w_2,w_3,w_4\}]$  has
transversal number at least two), a contradiction.
\end{proof}

\subsection{$\mathcal{G}^{2}_{6}$-colorable}
In this section we will prove the following statement.

\begin{lemma}\label{main-lemma-Ht-G2-color}
Suppose that $\mathcal{H}_{t}$ is $\mathcal{G}^2_6$-colorable.
Then $\widetilde{\mathcal{H}}_{i}$ is $\mathcal{G}^2_6$-colorable for all $0\le i \le t$.
In particular, $\mathcal{H}[W] = \widetilde{\mathcal{H}}_{0}$ is $\mathcal{G}^2_6$-colorable.
\end{lemma}

The following lemma, which will be used in the proof of Lemma~\ref{main-lemma-Ht-G2-color},
can be easily proved using a probabilistic argument. Its proof can be found in \cite{LMR1}.

Consider a 3-graph with $V(\mathcal{G})=[m]$ and pairwise disjoint sets $V_1,\ldots,V_{m}$.
The blowup $\mathcal{G}[V_1,\ldots,V_{m}]$ of $\mathcal{G}$ 
is obtained from $\mathcal{G}$ by replacing each vertex $j\in[m]$ with the set $V_j$ and
each edge $\{j_1,j_2,j_3\}\in \mathcal{G}$ with the complete $3$-partite $3$-graph
with vertex classes $V_{j_1}$, $V_{j_2}$, and $V_{j_3}$.
For a $3$-graph $\mathcal{H}$ we say that
a partition $V(\mathcal{H}) = \bigcup_{j\in[m]}V_j$ is a $\mathcal{G}$-coloring of $\mathcal{H}$
if $\mathcal{H} \subseteq \mathcal{G}[V_1,\ldots,V_{m}]$.

\begin{lemma}[\cite{LMR1}]\label{LEMMA:greedily-embedding-Gi}
Fix a real $\eta \in (0, 1)$ and integers $m, n\ge 1$.
Let $\mathcal{G}$ be a $3$-graph with vertex set~$[m]$ and let $\mathcal{H}$ be a further $3$-graph
with $v(\mathcal{H})=n$.
Consider a vertex partition $V(\mathcal{H}) = \bigcup_{i\in[m]}V_i$ and the associated
blowup $\widehat{\mathcal{G}} = \mathcal{G}[V_1,\ldots,V_{m}]$ of $\mathcal{G}$.
If two sets $T \subseteq [m]$ and $S\subseteq V$ (we allow $S$ to contain vertices from $V_i$ for $i\in T$)
have the properties
\begin{itemize}
\item[(a)] $|V_{j}| \ge (|S|+1)|T|\eta^{1/3} n + |S|$  for all $j \in T$,
\item[(b)] $|\mathcal{H}[V_{j_1},V_{j_2},V_{j_3}]| \ge |\widehat{\mathcal{G}}[V_{j_1},V_{j_2},V_{j_3}]| - \eta n^3$
            for all $\{j_1,j_2,j_3\} \in \binom{T}{3}$, and
\item[(c)] $|L_{\mathcal{H}}(v)[V_{j_1},V_{j_2}]| \ge |L_{\widehat{\mathcal{G}}}(v)[V_{j_1},V_{j_2}]| - \eta n^{3}$
            for all $v\in S$ and $\{j_1,j_{2}\} \in \binom{T}{2}$.
\end{itemize}
then there exists a selection of vertices $u_j\in V_j\setminus S$ for all $j\in [T]$
such that $U = \{u_j\colon j\in T\}$ satisfies
$\widehat{\mathcal{G}}[U] \subseteq \mathcal{H}[U]$ and
$L_{\widehat{\mathcal{G}}}(v)[U] \subseteq L_{\mathcal{H}}(v)[U]$ for all $v\in S$.
In particular, if $\mathcal{H} \subseteq \widehat{\mathcal{G}}$,
then $\widehat{\mathcal{G}}[U] = \mathcal{H}[U]$ and
$L_{\widehat{\mathcal{G}}}(v)[U] = L_{\mathcal{H}}(v)[U]$ for all $v\in S$.
\end{lemma}

Now we prove Lemma~\ref{main-lemma-Ht-G2-color}.

\begin{proof}[Lemma~\ref{main-lemma-Ht-G2-color}]
Similar to Lemma \ref{main-lemma-Ht-semibipartite},
the proof of Lemma \ref{main-lemma-Ht-G2-color} is also by backward induction on $i$,
and the base case is $i = t$ as $\widetilde{\mathcal{H}}_{t} = \mathcal{H}_t$.
Now suppose that $\widetilde{\mathcal{H}}_{i+1}$ is $\mathcal{G}^2_6$-colorable for some $0 \le i \le t-1$,
and we want to show that $\widetilde{\mathcal{H}}_{i}$ is also $\mathcal{G}^2_6$-colorable.

Since $\widetilde{\mathcal{H}}_{i+1}$ is $\mathcal{G}^2_6$-colorable,
let
\[
\mathcal{P} = \{V_{1}^{i+1}, \ldots, V_{6}^{i+1}\}.
\]
be the set of six parts in $\widetilde{\mathcal{H}}_{i+1}$ such that there is no edge between
$V_1^{i+1}V_{2}^{i+1}V_3^{i+1}$, $V_{1}^{i+1}V_{2}^{i+1}V_{6}^{i+1}$,
$V_{3}^{i+1}V_{4}^{i+1}V_{5}^{i+1}$, and $V_{4}^{i+1}V_{5}^{i+1}V_{6}^{i+1}$
(and every edge in $\widetilde{\mathcal{H}}_{i+1}$ hits at most one vertex in $V_{j}^{i+1}$ for every $j\in [6]$).
We may assume that each set $V_{j}^{i+1}$ is a union of some equivalence classes.

Let
\begin{align}
y = (y_1,\ldots,y_6)
= \left(|V_1^{i+1}|/\tilde{n}, \ldots,|V_6^{i+1}|/\tilde{n}\right), \notag
\end{align}
and notice that a similar argument as in the proof of Lemma \ref{lemma-blowup-lagrang} yields
\begin{align}\label{equ-tilde-Hi+1-p(y)}
|\widetilde{\mathcal{H}}_{i+1}| \le p_{\mathcal{G}^2_{6}}(y) \tilde{n}^{3}.
\end{align}

First we give a lower bound and an upper bound for the size of every set in $\mathcal{P}$.

\begin{claim}\label{claim-P-size-2}
We have $\left| |A| - {\tilde{n}}/{6} \right| < 20 \epsilon^{1/4} \tilde{n}$ for every set $A \in \mathcal{P}$.
\end{claim}
\begin{proof}[Proof of Claim \ref{claim-P-size-2}]
Let $\eta = 4\epsilon^{1/2}$ and note that by assumption $\eta>0$ is sufficiently small and $\tilde{n}$ is sufficiently large.
First, it follows from $(\ref{equ-tilde-Hi+1-p(y)})$ and Lemma~\ref{lemma-low-bound-tilde-Hi} that
\begin{align}
p_{\mathcal{G}^2_{6}}(y_1,\ldots,y_{6})
\ge \left(4/9- 10 \epsilon^{1/2} \right)\binom{\tilde{n}}{3}/\tilde{n}^3
\ge 2/27 - \eta. \notag
\end{align}
On the other hand,
let $a=(y_3+y_6)/2$, $b = (y_1+y_2)/2$, $c=(y_4+y_5)/2$, $d=(b+c)/2$ and
recall from the proof of Lemma~\ref{lemma-lagrang-G26} that
\begin{align}
p_{\mathcal{G}^2_{6}}(y_1,\ldots,y_{6})
& = y_3y_6(y_1+y_2+y_4+y_5) \notag\\
& \quad + (y_1+y_2)(y_3+y_6)(y_4+y_5) + y_1y_2(y_4+y_5) + y_4y_5(y_1+y_2) \notag\\
& \le 2a^2(b+c)+8abc+2bc(b+c)
 \le 4a^2d + 8ad^2 + 4d^3
 = 2\left((a+d)\cdot (a+d)\cdot 2d\right). \notag
\end{align}
Therefore,
\begin{align}\label{equ:G2-size-U-1}
(a+d)\cdot (a+d)\cdot 2d \ge 1/27 - \eta/2,
\end{align}
and
\begin{align}\label{equ:G2-size-U-2}
4d(a^2- y_3y_6) \le \eta, \quad
4d(d^2-bc) \le  \eta, \quad
2c(b^2-y_1y_2) \le  \eta, \quad
2b(c^2-y_4y_5) \le  \eta.
\end{align}
Now $(\ref{equ:G2-size-U-1})$ and $2a+4d=1$ yield
\begin{align}
\eta/2
\ge 1/27 - (a+d)^2\cdot 2d
= 1/27 - (1+2a)^2(1-2a)/32
& = (a-1/6)^2(a/4+5/24) \notag\\
& \ge (a-1/6)^2/8, \notag
\end{align}
whence $|a-1/6| \le 2\eta^{1/2}$.
By $2|a-1/6| = 4|d-1/6|$ this implies $|d-1/6|\le \eta^{1/2}$.
Since $\eta$ is sufficiently small, it follows that $a,d \ge 1/8$.
So the first inequality in  $(\ref{equ:G2-size-U-2})$ leads to $(y_3-y_6)\le 8\eta$,
whence $|y_3-y_6| \le 3\eta^{1/2}$.
By the triangle inequality we obtain
\begin{align}
2|y_3-1/6|
\le |y_3-y_6| + |y_3+y_6-1/3|
\le 3\eta^{1/2} + 2|a-1/6|
\le 7\eta^{1/2}, \notag
\end{align}
which shows $|y_3-1/6| \le 4\eta^{1/2}$.
Similarly, $|y_6-1/6|\le 4\eta^{1/2}$.
Applying the same reasoning to the other estimates in $(\ref{equ:G2-size-U-2})$
we obtain first $|b-1/6|, |c-1/6| \le 3\eta^{1/2}$ and then
$|y_i-1/6|\le 5\eta^{1/2}$ for every $i\in\{1,2,4,5\}$.
\end{proof}

\begin{figure}[htbp]
\centering
\begin{tikzpicture}[xscale=3,yscale=3]
    \node (a1) at (-0.5,0.5) {};
    \node (a2) at (0.5,0.5) {};
    \node (a3) at (0.5,-0.5) {};
    \node (a4) at (-0.5,-0.5) {};
    \node (b1) at (-1,0) {};
    \node (b2) at (1,0) {};
    \node (v) at (0,0) {};

    \fill (a1)  node {$U_1$};
    \fill (a2)  node {$U_2$};
    \fill (a3)  node {$U_4$};
    \fill (a4)  node {$U_5$};
    \fill (b1)  node {$U_6$};
    \fill (b2)  node {$U_3$};
    \fill (v)  node {$C_v$};

    \draw[line width=0.8pt] (0.5,0.5) circle [radius = 0.13];
    \draw[line width=0.8pt] (0.5,-0.5) circle [radius = 0.13];
    \draw[line width=0.8pt] (-0.5,0.5) circle [radius = 0.13];
    \draw[line width=0.8pt] (-0.5,-0.5) circle [radius = 0.13];
    \draw[line width=0.8pt] (1,0) circle [radius = 0.13];
    \draw[line width=0.8pt] (-1,0) circle [radius = 0.13];
    \draw[line width=0.8pt] (0,0) circle [radius = 0.13];

    \draw[line width=0.8pt,color=sqsqsq,fill=sqsqsq,fill opacity=0.15,dash pattern=on 1pt off 1.2pt]
        (0.5+0.15,0.5) to [out = 90, in = 0] (-0.5,0.5+0.15) to [out = 180, in = 180] (-1,0-0.15)
        to [out = 0, in = 180] (-0.5,0.5-0.15) to [out = 0, in = 270] (0.5+0.15,0.5);
    \draw[line width=0.8pt,color=sqsqsq,fill=sqsqsq,fill opacity=0.15,dash pattern=on 1pt off 1.2pt]
        (-0.5-0.15,0.5) to [out = 90, in = 180] (0.5,0.5+0.15) to [out = 0, in = 0] (1,0-0.15)
        to [out = 180, in = 0] (0.5,0.5-0.15) to [out = 180, in = 270] (-0.5-0.15,0.5);
    \draw[line width=0.8pt,color=sqsqsq,fill=sqsqsq,fill opacity=0.15,dash pattern=on 1pt off 1.2pt]
        (0.5+0.15,-0.5) to [out = 270, in = 0] (-0.5,-0.5-0.15) to [out = 180, in = 180] (-1,0+0.15)
        to [out = 0, in = 180] (-0.5,-0.5+0.15) to [out = 0, in = 90] (0.5+0.15,-0.5);
    \draw[line width=0.8pt,color=sqsqsq,fill=sqsqsq,fill opacity=0.15,dash pattern=on 1pt off 1.2pt]
        (-0.5-0.15,-0.5) to [out = 270, in = 180] (0.5,-0.5-0.15) to [out = 0, in = 0] (1,0+0.15)
        to [out = 180, in = 0] (0.5,-0.5+0.15) to [out = 180, in = 90] (-0.5-0.15,-0.5);

\end{tikzpicture}
\caption{$\widetilde{\mathcal{H}}_{i+1}$ is obtained from $\widetilde{\mathcal{H}}_{i}$ by symmetrizing $C_v$
to some equivalence class $C_u$ that is contained in $U_j$ for some $j\in[6]$.
Dashed lines indicate that there is no edge between these parts in $\widetilde{\mathcal{H}}_{i}$.}
\label{figure:Hi-G2}
\end{figure}
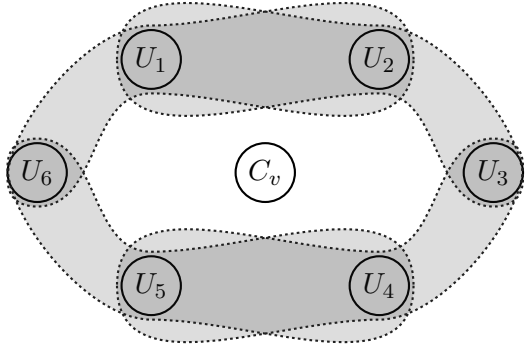

Lemma \ref{lemma-Hi+1-and-Hi} implies that either $\widetilde{\mathcal{H}}_{i} = \widetilde{\mathcal{H}}_{i+1}$
or there exists two equivalence classes $C_v$ and $C_u$ in $\widetilde{\mathcal{H}}_{i}$ such that $\widetilde{\mathcal{H}}_{i}$
is obtained from $\widetilde{\mathcal{H}}_{i+1}$ by symmetrizing $C_v$ to $C_u$ (see Figure~\ref{figure:Hi-G2}).
In the former case, there is nothing to prove, so we may assume that we are in the latter case.
Notice that $C_{v}$ and $C_u$ are contained in the same member in $\mathcal{P}$, and in particular,
Claim \ref{claim-P-size-2} implies that $|C_v| \le (1/6 + 10 \epsilon^{1/4})\tilde{n}$.
In the rest of the proof we will focus on the structure of $\widetilde{\mathcal{H}}_{i}$.
Let $U_j = V_{j}^{i+1} \setminus C_v$ for $j \in [6]$, $W' = W \setminus C_v$, and
\[
\mathcal{P}' = \{U_1,\ldots, U_6\}.
\]
Notice that there exists $j_0 \in [6]$ such that $U_{j_0}\cup C_v = V_{j_0}^{i+1}$,
and $U_j = V_{j}^{i+1}$ holds for all $j\in[6]\setminus\{j_0\}$.
In particular, no set in $\mathcal{P}'$ is the empty set.

First we will prove several claims about sets in $\mathcal{P}'$.
Since $U_1$ is a representative for sets in $\{U_1,U_2, U_4,U_5\}$
and $U_3$ is a representative for sets in $\{U_3,U_6\}$, we shall only prove the statements for $U_1$ and $U_3$,
and by symmetry, the statements hold for all sets in $\mathcal{P}'$.

Denote by $\widehat{\mathcal{G}}$ the blowup $\mathcal{G}_{6}^2[U_1,\ldots, U_6]$
of $\mathcal{G}_{6}^{2}$, and notice that $\widetilde{\mathcal{H}}_{i}[W'] \subset \widehat{\mathcal{G}}$.
For $j\in [6]$ fix a vertex $a_j \in U_j$,
let $\widetilde{G}_j = L_{\widehat{\mathcal{G}}}(a_j)$,
$G_j = \widetilde{G}_j[\{a_1,\ldots,a_6\}\setminus\{a_j\}]$,
and notice that $\widetilde{G}_j$ is a graph on $W'\setminus U_j$ and is a blowup of $G_j$ (see Figure~\ref{figure:G1-G3}).

\begin{figure}[htbp]
\centering
\subfigure[The graph ${G}_1$ is the $5$-vertex graph above, and $\widetilde{G}_1$ is a blowup of ${G}_1$.]
{
\begin{minipage}[t]{0.45\linewidth}
\centering
\begin{tikzpicture}[xscale=2.5,yscale=2.5]
    \node (a1) at (-0.5,0.5) {};
    \node (a2) at (0.5,0.5) {};
    \node (a3) at (0.5,-0.5) {};
    \node (a4) at (-0.5,-0.5) {};
    \node (b1) at (-1,0) {};
    \node (b2) at (1,0) {};

    \fill (a1) circle (0.02) node [above] {$a_1$};
    \fill (a2) circle (0.02) node [above] {$a_2$};
    \fill (a3) circle (0.02) node [below] {$a_4$};
    \fill (a4) circle (0.02) node [below] {$a_5$};
    \fill (b1) circle (0.02) node [left] {$a_6$};
    \fill (b2) circle (0.02) node [right] {$a_3$};

    \draw[line width=0.8pt] (0.5,0.5) circle [radius = 0.2];
    \draw[line width=0.8pt] (0.5,-0.5) circle [radius = 0.2];
    \draw[line width=0.8pt] (-0.5,0.5) circle [radius = 0.2];
    \draw[line width=0.8pt] (-0.5,-0.5) circle [radius = 0.2];
    \draw[line width=0.8pt] (1,0) circle [radius = 0.2];
    \draw[line width=0.8pt] (-1,0) circle [radius = 0.2];

    \draw[line width=0.8pt]
    (0.5,0.5) -- (0.5,-0.5);
    \draw[line width=0.8pt]
    (0.5,0.5) -- (-0.5,-0.5);
    \draw[line width=0.8pt]
    (1,0) -- (0.5,-0.5);
    \draw[line width=0.8pt]
    (1,0) -- (-0.5,-0.5);
    \draw[line width=0.8pt]
    (1,0) -- (-1,0);
    \draw[line width=0.8pt]
    (-1,0) -- (0.5,-0.5);
    \draw[line width=0.8pt]
    (-1,0) -- (-0.5,-0.5);
    \draw[line width=0.8pt]
    (-0.5,-0.5) -- (0.5,-0.5);

    \draw[line width=0.8pt,,dash pattern=on 1pt off 1.2pt]
    (-1.2,0.3) to [out = 0, in = 225] (0, 0.3)  to [out = 45, in = 250] (0.3, 0.9);

    \node at (-0.5,0.7+0.1) {$U_1$};
    \node at (0.5,0.7+0.1) {$U_2$};
    \node at (0.5,-0.7-0.1) {$U_4$};
    \node at (-0.5,-0.7-0.1) {$U_5$};
    \node at (-1,-0.2 -0.1) {$U_6$};
    \node at (1, -0.2-0.1) {$U_3$};
\end{tikzpicture}
\end{minipage}
}
\subfigure[The graph ${G}_3$ is the $5$-vertex graph above, and $\widetilde{G}_3$ is a blowup of ${G}_3$.]
{
\begin{minipage}[t]{0.45\linewidth}
\centering
\begin{tikzpicture}[xscale=2.5,yscale=2.5]
    \node (a1) at (-0.5,0.5) {};
    \node (a2) at (0.5,0.5) {};
    \node (a3) at (0.5,-0.5) {};
    \node (a4) at (-0.5,-0.5) {};
    \node (b1) at (-1,0) {};
    \node (b2) at (1,0) {};

    \fill (a1) circle (0.02) node [above] {$a_1$};
    \fill (a2) circle (0.02) node [above] {$a_2$};
    \fill (a3) circle (0.02) node [below] {$a_4$};
    \fill (a4) circle (0.02) node [below] {$a_5$};
    \fill (b1) circle (0.02) node [left] {$a_6$};
    \fill (b2) circle (0.02) node [right] {$a_3$};

    \draw[line width=0.8pt] (0.5,0.5) circle [radius = 0.2];
    \draw[line width=0.8pt] (0.5,-0.5) circle [radius = 0.2];
    \draw[line width=0.8pt] (-0.5,0.5) circle [radius = 0.2];
    \draw[line width=0.8pt] (-0.5,-0.5) circle [radius = 0.2];
    \draw[line width=0.8pt] (1,0) circle [radius = 0.2];
    \draw[line width=0.8pt] (-1,0) circle [radius = 0.2];

    \draw[line width=0.8pt]
    (-0.5,0.5) -- (0.5,-0.5);
    \draw[line width=0.8pt]
    (-0.5,0.5) -- (-0.5,-0.5);
    \draw[line width=0.8pt]
    (-0.5,0.5) -- (-1,0);
    \draw[line width=0.8pt]
    (-0.5,-0.5) -- (0.5,0.5);
    \draw[line width=0.8pt]
    (-0.5,-0.5) -- (-1,0);
    \draw[line width=0.8pt]
    (0.5,0.5) -- (0.5,-0.5);
    \draw[line width=0.8pt]
    (-1,0) -- (0.5,0.5);
    \draw[line width=0.8pt]
    (-1,0) -- (0.5,-0.5);

    \draw[line width=0.8pt,,dash pattern=on 1pt off 1.2pt]
    (0.75,0.75) -- (0.75,-0.75);

    \node at (-0.5,0.7+0.1) {$U_1$};
    \node at (0.5,0.7+0.1) {$U_2$};
    \node at (0.5,-0.7-0.1) {$U_4$};
    \node at (-0.5,-0.7-0.1) {$U_5$};
    \node at (-1,-0.2 -0.1) {$U_6$};
    \node at (1, -0.2-0.1) {$U_3$};

\end{tikzpicture}
\end{minipage}
}
\centering
\caption{Graphs ${G}_1$ and ${G}_3$.}
\label{figure:G1-G3}
\end{figure}
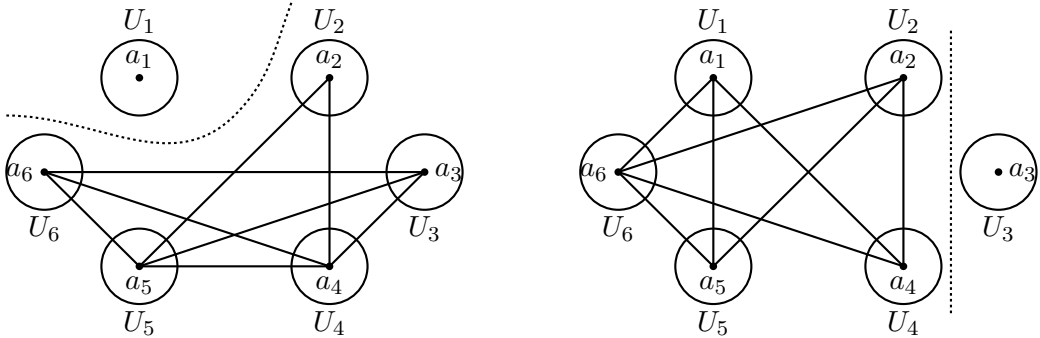

For every $w \in W$, let $L(w) = L_{\widetilde{\mathcal{H}}_{i}}(w)$ and $N(w) = N_{\widetilde{\mathcal{H}}_{i}}(w)$.
Since $\widetilde{\mathcal{H}}_{i+1}$ is $\mathcal{G}^2_6$-colorable and $\widetilde{\mathcal{H}}_{i}[W'] = \widetilde{\mathcal{H}}_{i+1}[W']$,
it follows that $L(w)[W'] \subset \widetilde{G}_{j}$ for all $j\in [6]$ and $w \in U_j$.
For every $j\in [6]$ and every $w \in U_j$ let
\[
M(w) = \left\{ w_1w_2 \in \widetilde{G}_{j} : w_1w_2 \not\in  L(w)[W'] \right\},
\]
and call members in $M(w)$ missing edges of $L(w)[W']$.

\begin{claim}\label{claim-size-MU(w)-MV(w')}
We have $|M(w)| \le 30 \epsilon^{1/4} \tilde{n}^2$ for every $w \in W$.
\end{claim}
\begin{proof}[Proof of Claim \ref{claim-size-MU(w)-MV(w')}]
We shall only prove the case $w\in U_1$, since the arguments for other cases are similar.
Fix a vertex $w \in U_1$.
Let $\widehat{G}_{1}$ be the blowup of $G_{1}$ obtained by replacing each vertex in $V(G_{1})$ with the set in $\mathcal{P}$
that contains it.
Since $\widetilde{\mathcal{H}}_{i+1}$ is $\mathcal{G}^{2}_6$-colorable,
$L_{\widetilde{\mathcal{H}}_{i+1}}(w) \subset \widehat{G}_1$.
On the other hand, since $L_{\widetilde{\mathcal{H}}_{i}}(w)[W'] = L_{\widetilde{\mathcal{H}}_{i+1}}(w)[W']$,
it follows from Lemma \ref{lemma-low-bound-tilde-Hi} and Claim \ref{claim-P-size-2} that
\begin{align}
|M(w)| = |\widetilde{G}_1 \setminus L_{\widetilde{\mathcal{H}}_{i}}(w)[W']|
&  \le |\widehat{G}_1 \setminus L_{\widetilde{\mathcal{H}}_{i+1}}(w)|  \notag \\
& = |\widehat{G}_{1}| - |L_{\widetilde{\mathcal{H}}_{i+1}}(w)| \notag \\
& < 8 \left( {1}/{6} + 10 \epsilon^{{1}/{4}} \right)^2 \tilde{n}^2
- \left({4}/{9} - 10\epsilon^{1/2} \right)\binom{\tilde{n}-1}{2} \notag \\
& < 30 \epsilon^{{1}/{4}}\tilde{n}^2. \notag
\end{align}
\end{proof}

By Lemma~\ref{lemma-low-bound-tilde-Hi} and Claim~\ref{claim-size-MU(w)-MV(w')},
$\widetilde{\mathcal{H}}_{i}$ and $\widehat{\mathcal{G}}$ satisfy the following statements,
which will be useful later when we applying Lemma~\ref{LEMMA:greedily-embedding-Gi}.

\begin{itemize}
\item[(a)] $|\widetilde{\mathcal{H}}_{i}[A_1,A_2,A_3]| \ge |\widehat{\mathcal{G}}[A_1,A_2,A_3]| - 2\epsilon^{1/2} n^3$
            for every triple $\{A_1,A_2,A_3\} \subset \mathcal{P}'$, and
\item[(b)] $|L_{\widetilde{\mathcal{H}}_{i}}(u)[A_1,A_2]| \ge |L_{\widehat{\mathcal{G}}}(u)[A_1,A_2]| - 30\epsilon^{1/4} n^{3}$
            for every $u\in W'$ and every pair $\{A_1,A_2\} \subset \mathcal{P}'$ satisfying $u\not\in A_1\cup A_2$.
\end{itemize}

\begin{claim}\label{claim-N(w)-cap-P}
Let $j\in[6]$ and $w\in U_j$.
Then $|N(w) \cap (W'\setminus U_j)| > |W' \setminus U_j| - 400 \epsilon^{1/4} \tilde{n}$.
\end{claim}
\begin{proof}[Proof of Claim~\ref{claim-N(w)-cap-P}]
We shall only prove the case that $P=U_1$, since the arguments for other cases are similar.
Let $w \in U_1$ and $W'' = W' \setminus U_1$.
Since $C_v$ is contained in exactly one set in $\mathcal{P}$, it follows from Claim~\ref{claim-P-size-2} that
all but at most one set in $\mathcal{P}'$ have size at least $\left({1}/{6} -  10 \epsilon^{{1}/{4}} \right)\tilde{n}$.
On the other hand, since $\delta(G_1) \ge 2$ and $\widetilde{G}_{1}$ is a blowup of $G_1$, we obtain
\[
\delta(\widetilde{G}_{1}) > \left({1}/{6} -  10 \epsilon^{{1}/{4}} \right)\tilde{n}.
\]
So it follows from Claim \ref{claim-size-MU(w)-MV(w')} that the number of vertices in $W''$ with degree $0$ in $L(w)[W']$ is at most
\[
\frac{ 2|M_{U}(w)| }{ \delta(\widetilde{G}_{1}) } <
\frac{ 60 \epsilon^{1/4} \tilde{n}^2}{ (1/6 - 10 \epsilon^{1/4})\tilde{n} } < 400 \epsilon^{1/4} \tilde{n}.
\]
\end{proof}

Recall that $\widetilde{\mathcal{H}}_{i+1}$ is obtained from $\widetilde{\mathcal{H}}_{i}$ by symmetrizing $C_v$ to $C_u$,
where $C_v$ and $C_u$ are equivalence classes of $v$ and $u$ in $\widetilde{\mathcal{H}}_{i}$, respectively.
Let $P_u$ denote the member in $\mathcal{P}'$ that contains $u$ and
notice that $P_{u} \cup C_v$ is a member in $\mathcal{P}$.
So Claim~\ref{claim-P-size-2} implies that $|P_{u} \cup C_v| \le (1/6 + 10 \epsilon^{1/4})\tilde{n}$.

\begin{claim}\label{claim-if-Cv-large}
Suppose that $|C_{v}| > \tilde{n}/12$.
Then every vertex in $W'\setminus P_{u}$ is adjacent to all vertices in $C_{v}$ in $\widetilde{\mathcal{H}}_{i}$.
\end{claim}
\begin{proof}[Proof of Claim~\ref{claim-if-Cv-large}]
We shall only prove the case $P_{u}=U_2$, since the arguments for other cases are similar.
First it follows from $|P_{u} \cup C_{v}| \le (1/6 + 10 \epsilon^{1/4})\tilde{n}$ that
$|P_{u}| <  (1/6 + 10 \epsilon^{1/4})\tilde{n} - \tilde{n}/12 < \tilde{n}/10$.
Let $w \in W' \setminus P_{u}$, and suppose that $w$ is not adjacent to any vertex in $C_{v}$.
We shall only prove that case $w \in U_1$, since the arguments for other cases are similar.

Since $\widetilde{\mathcal{H}}_{i}[W'] = \widetilde{\mathcal{H}}_{i+1}[W']$ and
$\widetilde{\mathcal{H}}_{i+1}$ is $\mathcal{G}^{2}_{6}$-colorable,
$L_{\widetilde{\mathcal{H}}_{i}}(w)[W'] \subset \widetilde{G}_{1}$.
On the other hand,
since $N_{\widetilde{\mathcal{H}}_{i}}(w) \cap C_v = \emptyset$,
we actually have $L_{\widetilde{\mathcal{H}}_{i}}(w) \subset \widetilde{G}_{1}$.
It follows from the definition of $\widetilde{G}_{1}$, Claim \ref{claim-P-size-2},
and $|U_2| = |P_{u}| < \tilde{n}/10$ that
\begin{align}
|L_{\widetilde{\mathcal{H}}_{i}}(w)|
\le |\widetilde{G}_{1}|
< 6 \left( {1}/{6} + 10 \epsilon^{\frac{1}{4}} \right)^2 \tilde{n}^2
+ 2\times \frac{\tilde{n}}{10}\left( {1}/{6} + 10 \epsilon^{\frac{1}{4}} \right) \tilde{n}
< \left( {2}/{9} - 10 {\epsilon}^{1/2} \right)\tilde{n}^2, \notag
\end{align}
which contradicts Lemma \ref{lemma-low-bound-tilde-Hi}.

Therefore, $w$ is adjacent to some vertex in $C_v$ (in $\widetilde{\mathcal{H}}_{i}$).
Since $C_{v}$ is an equivalence class in $\widetilde{\mathcal{H}}_{i}$, $w$ is adjacent to all vertices in $C_{v}$
(in $\widetilde{\mathcal{H}}_{i}$).
\end{proof}

\begin{figure}[htbp]
\centering
\begin{tikzpicture}[xscale=3,yscale=3]
\node (v) at (-1,0.5) {};
\fill (v) circle (0.02) node [left] {$v$};
\node (w1) at (-0.4,0.6) {};
\fill (w1) circle (0.02) node at (-0.4+0.12,0.6+0.05) {$w_1$};
\node (w2) at (-0.5,0.4) {};
\fill (w2) circle (0.02) node at (-0.6,0.5) {$w_1'$};
\node (w3) at (0.5,0.5) {};
\fill (w3) circle (0.02) node [above] {$w_2$};
\node (w4) at (-1,0) {};
\fill (w4) circle (0.02) node [left] {$w_6$};
\node (w5) at (1,0) {};
\fill (w5) circle (0.02) node [right] {$w_3$};
\node (w6) at (0.5,-0.5) {};
\fill (w6) circle (0.02) node [below] {$w_4$};
\node (w7) at (-0.5,-0.5) {};
\fill (w7) circle (0.02) node [below] {$w_5$};

\draw[line width=0.5pt] (-1,0.5) -- (-0.5,0.4);
\draw[line width=0.5pt] (-1,0.5) -- (-0.4,0.6);

\draw[line width=0.5pt] (-0.4,0.6) -- (-0.5,0.4);
\draw[line width=0.5pt] (-0.4,0.6) -- (0.5,0.5);
\draw[line width=0.5pt] (-0.4,0.6) -- (-0.5,-0.5);
\draw[line width=0.5pt] (-0.4,0.6) -- (0.5,-0.5);
\draw[line width=0.5pt] (-0.4,0.6) -- (1,0);
\draw[line width=0.5pt] (-0.4,0.6) -- (-1,0);
\draw[line width=0.5pt] (-0.5,0.4) -- (0.5,0.5);
\draw[line width=0.5pt] (-0.5,0.4) -- (-0.5,-0.5);
\draw[line width=0.5pt] (-0.5,0.4) -- (0.5,-0.5);
\draw[line width=0.5pt] (-0.5,0.4) -- (1,0);
\draw[line width=0.5pt] (-0.5,0.4) -- (-1,0);
\draw[line width=0.5pt] (0.5,0.5) -- (-0.5,-0.5);
\draw[line width=0.5pt] (0.5,0.5) -- (0.5,-0.5);
\draw[line width=0.5pt] (0.5,0.5) -- (1,0);
\draw[line width=0.5pt] (0.5,0.5) -- (-1,0);
\draw[line width=0.5pt] (-0.5,-0.5) -- (0.5,-0.5);
\draw[line width=0.5pt] (-0.5,-0.5) -- (1,0);
\draw[line width=0.5pt] (-0.5,-0.5) -- (-1,0);
\draw[line width=0.5pt] (0.5,-0.5) -- (1,0);
\draw[line width=0.5pt] (0.5,-0.5) -- (-1,0);
\draw[line width=0.5pt] (1,0) -- (-1,0);
    \draw[line width=0.8pt] (0.5,0.5) circle [radius = 0.2];
    \draw[line width=0.8pt] (0.5,-0.5) circle [radius = 0.2];
    \draw[line width=0.8pt] (-0.5,0.5) circle [radius = 0.2];
    \draw[line width=0.8pt] (-0.5,-0.5) circle [radius = 0.2];
    \draw[line width=0.8pt] (1,0) circle [radius = 0.2];
    \draw[line width=0.8pt] (-1,0) circle [radius = 0.2];
    \node at (-0.5,0.65+0.15) {$U_1$};
    \node at (0.5,0.65+0.15) {$U_2$};
    \node at (0.5,-0.65-0.15) {$U_4$};
    \node at (-0.5,-0.65-0.15) {$U_5$};
    \node at (-1.15-0.15,0) {$U_6$};
    \node at (1.15+0.15,0) {$U_3$};
\end{tikzpicture}
\caption{The $3$-graph
$F = \widetilde{\mathcal{H}}_i[\{w_1,w_2,\ldots,w_6\}] \cup \widetilde{\mathcal{H}}_i[\{w_1',w_2,\ldots,w_6\}] \cup \{vw_1w_1'\}$
is a member in $M_2$ with core $\{w_1,w_1', w_2,\ldots,w_6\}$.
In particular, $\tau(\{w_1w_3w_4,w_1'w_5w_6\}) > 1$.}
\label{figure:no-link-inside-Uj-1}
\end{figure}
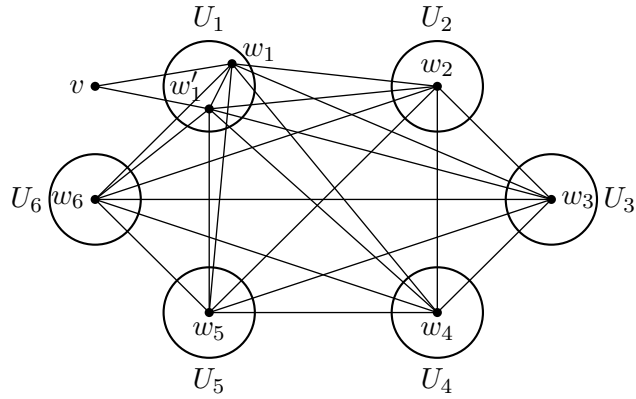

\begin{claim}\label{claim-Lv[S]-empty}
We have $L(v)[A] = \emptyset$ for every set $A \in \mathcal{P}'$.
\end{claim}
\begin{proof}[Proof of Claim~\ref{claim-Lv[S]-empty}]
Suppose to the contrary that there exists an edge $w_1w_1' \in L_{\widetilde{\mathcal{H}}_i}(v)[A]$ for some $A \in \mathcal{P}'$.
We shall only prove the case $A = U_1$, since the arguments for other cases are similar.
It follows from Claim \ref{claim-N(w)-cap-P} that
\begin{align}\label{equ-N(w1)-cap-N(w2)-cap-W'}
|N(w_1) \cap N(w_1') \cap (W'\setminus U_1)| > |W' \setminus U_1| - 800 \epsilon^{1/4}\tilde{n}.
\end{align}

Suppose that $|W' \setminus U_1| > 11\tilde{n}/15$.
Then by Claim \ref{claim-P-size-2},
$|U_{j}| \ge 11\tilde{n}/15 - 4(1/6+20\epsilon^{1/4})\tilde{n}> \tilde{n}/20$ for every $j\in [2,6]$.
Applying Lemma~\ref{LEMMA:greedily-embedding-Gi} with $S = \{w_1,w_1'\}$, $T = [2,6]$, and $\eta = 30\epsilon^{1/4}$
we obtain $w_j\in U_j$ for $j\in[2,6]$ (see Figure~\ref{figure:no-link-inside-Uj-1}) such that
the induced subgraphs of $\widetilde{\mathcal{H}}_i$ on sets $\{w_1,w_2,\ldots,w_6\}$ and $\{w_1',w_2,\ldots,w_6\}$
are isomorphic to $\mathcal{G}_{6}^{2}$.
Let $F = \widetilde{\mathcal{H}}_i[\{w_1,w_2,\ldots,w_6\}] \cup \widetilde{\mathcal{H}}_i[\{w_1',w_2,\ldots,w_6\}] \cup \{vw_1w_1'\}$.
Then it is easy to see that $F\in M_2$ with core $\{w_1,w_1',w_2,\ldots,w_6\}$ (see Figure~\ref{figure:no-link-inside-Uj-1}), a contradiction.

\begin{figure}[htbp]
\centering
\begin{tikzpicture}[xscale=3,yscale=3]
\node (v) at (-0.3,0.1) {};
\fill (v) circle (0.02) node [below] {$v$};
\node (w1) at (-0.4,0.6) {};
\fill (w1) circle (0.02) node at (-0.4+0.12,0.6+0.05) {$w_1$};
\node (w2) at (-0.5,0.4) {};
\fill (w2) circle (0.02) node at (-0.6,0.5) {$w_1'$};
\node (w4) at (-1,0) {};
\fill (w4) circle (0.02) node [left] {$w_6$};
\node (w5) at (1,0) {};
\fill (w5) circle (0.02) node [right] {$w_3$};
\node (w6) at (0.5,-0.5) {};
\fill (w6) circle (0.02) node [below] {$w_4$};
\node (w7) at (-0.5,-0.5) {};
\fill (w7) circle (0.02) node [below] {$w_5$};

\draw[line width=0.5pt] (-0.3,0.1) -- (-0.5,0.4);
\draw[line width=0.5pt] (-0.3,0.1) -- (-0.4,0.6);

\draw[line width=0.5pt] (-0.4,0.6) -- (-0.5,0.4);
\draw[line width=0.5pt] (-0.4,0.6) -- (-0.3,0.1);
\draw[line width=0.5pt] (-0.4,0.6) -- (-0.5,-0.5);
\draw[line width=0.5pt] (-0.4,0.6) -- (0.5,-0.5);
\draw[line width=0.5pt] (-0.4,0.6) -- (1,0);
\draw[line width=0.5pt] (-0.4,0.6) -- (-1,0);
\draw[line width=0.5pt] (-0.5,0.4) -- (-0.3,0.1);
\draw[line width=0.5pt] (-0.5,0.4) -- (-0.5,-0.5);
\draw[line width=0.5pt] (-0.5,0.4) -- (0.5,-0.5);
\draw[line width=0.5pt] (-0.5,0.4) -- (1,0);
\draw[line width=0.5pt] (-0.5,0.4) -- (-1,0);
\draw[line width=0.5pt] (-0.3,0.1) -- (-0.5,-0.5);
\draw[line width=0.5pt] (-0.3,0.1) -- (0.5,-0.5);
\draw[line width=0.5pt] (-0.3,0.1) -- (1,0);
\draw[line width=0.5pt] (-0.3,0.1) -- (-1,0);
\draw[line width=0.5pt] (-0.5,-0.5) -- (0.5,-0.5);
\draw[line width=0.5pt] (-0.5,-0.5) -- (1,0);
\draw[line width=0.5pt] (-0.5,-0.5) -- (-1,0);
\draw[line width=0.5pt] (0.5,-0.5) -- (1,0);
\draw[line width=0.5pt] (0.5,-0.5) -- (-1,0);
\draw[line width=0.5pt] (1,0) -- (-1,0);
    \draw[line width=0.8pt] (0.5,0.5) circle [radius = 0.2];
    \draw[line width=0.8pt] (0.5,-0.5) circle [radius = 0.2];
    \draw[line width=0.8pt] (-0.5,0.5) circle [radius = 0.2];
    \draw[line width=0.8pt] (-0.5,-0.5) circle [radius = 0.2];
    \draw[line width=0.8pt] (1,0) circle [radius = 0.2];
    \draw[line width=0.8pt] (-1,0) circle [radius = 0.2];
    \node at (-0.5,0.65+0.15) {$U_1$};
    \node at (0.5,0.65+0.15) {$U_2$};
    \node at (0.5,-0.65-0.15) {$U_4$};
    \node at (-0.5,-0.65-0.15) {$U_5$};
    \node at (-1.15-0.15,0) {$U_6$};
    \node at (1.15+0.15,0) {$U_3$};
\end{tikzpicture}
\caption{The $3$-graph
$F = \widetilde{\mathcal{H}}_i[\{w_1,w_3,\ldots,w_6\}] \cup \widetilde{\mathcal{H}}_i[\{w_1',w_3,\ldots,w_6\}] \cup \{vw_1w_1'\}
\cup \{e_j\colon j\in[3,6]\}$
is a member in $M_2$ with core $\{v,w_1,w_1', w_3,\ldots,w_6\}$.
In particular, $\tau(\{w_1w_3w_4,w_1'w_5w_6\}) > 1$.}
\label{figure:no-link-inside-Uj-2}
\end{figure}

Suppose that $|W' \setminus U_1| \le 11\tilde{n}/15 \le 5(1/6 - 10 \epsilon^{1/4})\tilde{n}$.
Then by Claim \ref{claim-P-size-2}, $|C_v| \ge \tilde{n} - (1/6 + 10 \epsilon^{1/4})\tilde{n} - 11\tilde{n}/15 > \tilde{n}/12$ and $P_{u} \neq U_1$.
We shall only prove that case $P_{u} = U_2$, since the arguments for other cases are similar.
Applying Lemma~\ref{LEMMA:greedily-embedding-Gi} with $S = \{w_1,w_1'\}$, $T = [3,6]$, and $\eta = 30\epsilon^{1/4}$
we obtain $w_j \in U_j$ for $j\in[3,6]$ (see Figure~\ref{figure:no-link-inside-Uj-2})
such that the induced subgraphs of $\widetilde{\mathcal{H}}_i$ and $\widehat{\mathcal{G}}$
on the sets $\{w_1,w_3,\ldots,w_6\}$ and $\{w_1', w_3, \ldots, w_6\}$ are isomorphic (and they are all $2$-covered), respectively.
For $j\in [3,6]$ let $e_j \in \widetilde{\mathcal{H}}_i$ be an edge containing $v$ and $w_j$
(by Claim~\ref{claim-if-Cv-large}, $v$ is adjacent to $w_j$, so such $e_j$ exists).
Define $F = \widetilde{\mathcal{H}}_i[\{w_1,w_3,\ldots,w_6\}] \cup \widetilde{\mathcal{H}}_i[\{w_1',w_3,\ldots,w_6\}] \cup \{vw_1w_1'\}
\cup \{e_j\colon j\in[3,6]\}$.
Then it is easy to see that $F\in M_2$ with core $\{v,w_1,w_1',w_3,\ldots,w_6\}$ (see Figure~\ref{figure:no-link-inside-Uj-2}), a contradiction.
\end{proof}

\begin{claim}\label{claim-less-one-small-intersect}
There is at most one set $A \in \mathcal{P}'$ such that $|N(v) \cap A| < \tilde{n}/48$.
\end{claim}
\begin{proof}[Proof of Claim \ref{claim-less-one-small-intersect}]
Let $U'_j = N(v) \cap U_j$ for $j\in[6]$.
By Claim \ref{claim-Lv[S]-empty}, $L(v)$ is a $6$-partite graph (not necessarily complete)
with the set of parts $\mathcal{P}'' := \{U'_1,U'_2,U'_3,U'_4,V'_1,V'_2\}$.
Suppose to the contrary that there are at least two sets in $\mathcal{P}''$ that have size at most $\tilde{n}/48$.
Then, by Claim \ref{claim-P-size-2},
\[
|L(v)|
 \le 6 \left( {1}/{6} + 10 \epsilon^{{1}/{4}} \right)^2 \tilde{n}^2 + \left( {\tilde{n}}/{48} \right)^{2}
+ 8 \times {\tilde{n}}/{48} \times \left( {1}/{6} + 10 \epsilon^{{1}/{4}} \right) \tilde{n}
 < \left({2}/{9} -10 {\epsilon}^{1/2} \right)\tilde{n}^2,
\]
which contradicts Lemma \ref{lemma-low-bound-tilde-Hi}.
\end{proof}

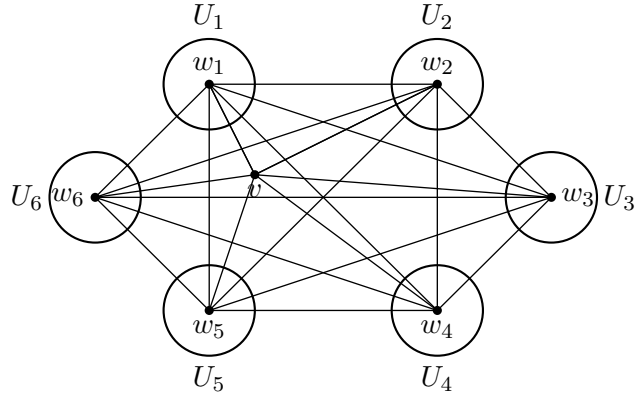
\begin{figure}[htbp]
\centering
\begin{tikzpicture}[xscale=3,yscale=3]
\node (v) at (-0.3,0.1) {};
\fill (v) circle (0.02) node [below] {$v$};
\node (w1) at (-0.5,0.5) {};
\fill (w1) circle (0.02) node [above] {$w_1$};
\node (w2) at (0.5,0.5) {};
\fill (w2) circle (0.02) node [above] {$w_2$};
\node (w4) at (-1,0) {};
\fill (w4) circle (0.02) node [left] {$w_6$};
\node (w5) at (1,0) {};
\fill (w5) circle (0.02) node [right] {$w_3$};
\node (w6) at (0.5,-0.5) {};
\fill (w6) circle (0.02) node [below] {$w_4$};
\node (w7) at (-0.5,-0.5) {};
\fill (w7) circle (0.02) node [below] {$w_5$};

\draw[line width=0.5pt] (-0.3,0.1) -- (-0.5,0.5);
\draw[line width=0.5pt] (-0.3,0.1) -- (0.5,0.5);

\draw[line width=0.5pt] (0.5,0.5) -- (-0.5,0.5);
\draw[line width=0.5pt] (0.5,0.5) -- (-0.3,0.1);
\draw[line width=0.5pt] (0.5,0.5) -- (-0.5,-0.5);
\draw[line width=0.5pt] (0.5,0.5) -- (0.5,-0.5);
\draw[line width=0.5pt] (0.5,0.5) -- (1,0);
\draw[line width=0.5pt] (0.5,0.5) -- (-1,0);
\draw[line width=0.5pt] (-0.5,0.5) -- (-0.3,0.1);
\draw[line width=0.5pt] (-0.5,0.5) -- (-0.5,-0.5);
\draw[line width=0.5pt] (-0.5,0.5) -- (0.5,-0.5);
\draw[line width=0.5pt] (-0.5,0.5) -- (1,0);
\draw[line width=0.5pt] (-0.5,0.5) -- (-1,0);
\draw[line width=0.5pt] (-0.3,0.1) -- (-0.5,-0.5);
\draw[line width=0.5pt] (-0.3,0.1) -- (0.5,-0.5);
\draw[line width=0.5pt] (-0.3,0.1) -- (1,0);
\draw[line width=0.5pt] (-0.3,0.1) -- (-1,0);
\draw[line width=0.5pt] (-0.5,-0.5) -- (0.5,-0.5);
\draw[line width=0.5pt] (-0.5,-0.5) -- (1,0);
\draw[line width=0.5pt] (-0.5,-0.5) -- (-1,0);
\draw[line width=0.5pt] (0.5,-0.5) -- (1,0);
\draw[line width=0.5pt] (0.5,-0.5) -- (-1,0);
\draw[line width=0.5pt] (1,0) -- (-1,0);
    \draw[line width=0.8pt] (0.5,0.5) circle [radius = 0.2];
    \draw[line width=0.8pt] (0.5,-0.5) circle [radius = 0.2];
    \draw[line width=0.8pt] (-0.5,0.5) circle [radius = 0.2];
    \draw[line width=0.8pt] (-0.5,-0.5) circle [radius = 0.2];
    \draw[line width=0.8pt] (1,0) circle [radius = 0.2];
    \draw[line width=0.8pt] (-1,0) circle [radius = 0.2];
    \node at (-0.5,0.65+0.15) {$U_1$};
    \node at (0.5,0.65+0.15) {$U_2$};
    \node at (0.5,-0.65-0.15) {$U_4$};
    \node at (-0.5,-0.65-0.15) {$U_5$};
    \node at (-1.15-0.15,0) {$U_6$};
    \node at (1.15+0.15,0) {$U_3$};
\end{tikzpicture}
\caption{The $3$-graph $F = \widetilde{\mathcal{H}}_i[\{w_1,\ldots,w_6\}] \cup \{e_j\colon j\in[6]\}$ is
a member in $M_2$ with core $\{v,w_1,\ldots,w_6\}$.
In particular, $\tau(\{w_1w_3w_4,w_2w_5w_6\}) > 1$.}
\label{figure:v-has-no-neighbor-in-some-Uj}
\end{figure}

\begin{claim}\label{claim-N(v)-cap-S-empty}
There exists a set $A \in \mathcal{P}'$ such that $N(v) \cap A = \emptyset$.
\end{claim}
\begin{proof}[Proof of Claim \ref{claim-N(v)-cap-S-empty}]
Suppose to the contrary that every set $A\in \mathcal{P}'$ satisfies $A \cap N(v) \neq \emptyset$.
By Claim \ref{claim-less-one-small-intersect},
there are at least five sets $A' \in \mathcal{P}'$ with $|A' \cap N(v)| \ge \tilde{n}/48$.
We shall only prove the case that
every set $A' \in \mathcal{P}'\setminus\{U_1\}$ satisfies $|A' \cap N(v)| \ge \tilde{n}/48$,
since the arguments for other cases are similar.

Fix a vertex $w_1 \in N(v) \cap U_1$.
Let $U_j' = U_j \cap N(v)$ for $i\in [2,6]$.
By assumption, $|U_j'| \ge \tilde{n}/48$ for $j\in[2,6]$.
So applying Lemma~\ref{LEMMA:greedily-embedding-Gi} with $T = \{w_1\}$, $S= [2,6]$, and $\eta = 30\epsilon^{1/4}$
we obtain $w_j \in U_j'$ for $j\in [2,6]$ (see Figure~\ref{figure:v-has-no-neighbor-in-some-Uj})
such that the induced subgraph of $\widetilde{\mathcal{H}}_i$
on $\{w_1,\ldots,w_6\}$ is isomorphic to $\mathcal{G}_6^2$.
For $j\in [6]$ let $e_j \in \widetilde{\mathcal{H}}_i$ be an edge containing $v$ and $w_j$.
Define $F = \widetilde{\mathcal{H}}_i[\{w_1,\ldots,w_6\}] \cup \{e_j\colon j\in[6]\}$.
Then it is easy to see that $F$ is a member in $M_2$ with core $\{v,w_1,\ldots,w_6\}$ (see Figure~\ref{figure:v-has-no-neighbor-in-some-Uj}),
a contradiction.
\end{proof}

Our next step is to show that $\widetilde{\mathcal{H}}_{i}$ is $\mathcal{G}_{6}^{2}$-colorable with
the sets of parts $\widetilde{\mathcal{P}}$,
where $\widetilde{\mathcal{P}}$ is obtained from $\mathcal{P}'$ by replacing $A$ with $A \cup C_v$
and the set $A$ is guaranteed by Claim \ref{claim-N(v)-cap-S-empty}.
We shall only prove the case $A = U_1$, since the arguments for other cases are similar.

Let
\begin{align}
B_{v} = \left\{ ww' \in L(v): ww' \not\in \widetilde{G}_{1} \right\}, \quad{\rm and}\quad
M_{v} = \left\{ ww' \in \widetilde{G}_{1}: ww' \not\in L(v) \right\}. \notag
\end{align}
Members in $B_v$ are called bad edges of $L(v)$ and members in $M_v$ are called missing edges of $L(v)$.

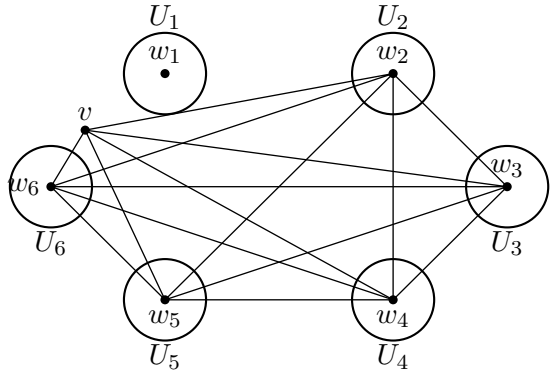
\begin{figure}[htbp]
\centering
\begin{tikzpicture}[xscale=3,yscale=3]
    \node (v) at (-0.85,0.25) {};
    \node (w6) at (-1,0) {};
    \node (w2) at (0.5,0.5) {};
    \node (w3) at (1,0) {};
    \node (w4) at (0.5,-0.5) {};
    \node (w5) at (-0.5,-0.5) {};
    \node (w1) at (-0.5,0.5) {};

    \fill (v) circle (0.02) node [above] {$v$};
    \fill (w6) circle (0.02) node [left] {$w_6$};
    \fill (w2) circle (0.02) node [above] {$w_2$};
    \fill (w3) circle (0.02) node [above] {$w_3$};
    \fill (w4) circle (0.02) node [below] {$w_4$};
    \fill (w5) circle (0.02) node [below] {$w_5$};
    \fill (w1) circle (0.02) node [above] {$w_1$};

    \draw[line width=0.8pt] (0.5,0.5) circle [radius = 0.18];
    \draw[line width=0.8pt] (0.5,-0.5) circle [radius = 0.18];
    \draw[line width=0.8pt] (-0.5,0.5) circle [radius = 0.18];
    \draw[line width=0.8pt] (-0.5,-0.5) circle [radius = 0.18];
    \draw[line width=0.8pt] (1,0) circle [radius = 0.18];
    \draw[line width=0.8pt] (-1,0) circle [radius = 0.18];

    \draw[line width=0.5pt] (-0.85, 0.25) -- (0.5,0.5) -- (-1,0) -- (-0.85,0.25);
    \draw[line width=0.5pt] (1,0) -- (0.5,-0.5) -- (-0.5,-0.5) -- (1,0);
    \draw[line width=0.5pt] (-0.85, 0.25) -- (1,0);
    \draw[line width=0.5pt] (-0.85, 0.25) -- (0.5,-0.5);
    \draw[line width=0.5pt] (-0.85, 0.25) -- (-0.5,-0.5);
    \draw[line width=0.5pt] (0.5, 0.5) -- (1,0);
    \draw[line width=0.5pt] (0.5, 0.5) -- (0.5,-0.5);
    \draw[line width=0.5pt] (0.5, 0.5) -- (-0.5,-0.5);
    \draw[line width=0.5pt] (-1,0) -- (1,0);
    \draw[line width=0.5pt] (-1,0) -- (0.5,-0.5);
    \draw[line width=0.5pt] (-1,0) -- (-0.5,-0.5);

    \node at (-0.5,0.65+0.1) {$U_1$};
    \node at (0.5,0.65+0.1) {$U_2$};
    \node at (0.5,-0.65-0.1) {$U_4$};
    \node at (-0.5,-0.65-0.1) {$U_5$};
    \node at (-1,0-0.25) {$U_6$};
    \node at (1,0-0.25) {$U_3$};
\end{tikzpicture}
\caption{The $3$-graph
$F = \widetilde{\mathcal{H}}_i[\{w_1,\ldots,w_6\}] \cup \{e_j\colon j \in\{4,5,6\}\} \cup \{vw_2w_3\}$
is a member in $M_3$ with core $\{v,w_1,\ldots,w_5\}$.}
\label{figure:Bv-small}
\end{figure}

\begin{claim}\label{CLAIM:bad-edges-are-few}
We have $|B_v| < 300\epsilon^{1/12} \tilde{n}^2$.
\end{claim}
\begin{proof}[Proof of Claim~\ref{CLAIM:bad-edges-are-few}]
Suppose to the contrary that $|B_v| \ge 300\epsilon^{1/12} \tilde{n}^2$.
Notice that every edge in $B_v$ must have one vertex in $U_2$ and the other vertex in $U_3 \cup U_6$.
By symmetry and the Pigeonhole principle, we may assume that at least $|B_v|/2$ edges in $B_v$
have one vertex in $U_2$ and the other vertex in $U_3$.
Then Claim~\ref{claim-P-size-2} and an easy averaging argument show that
there exists a vertex $w_2\in U_2$ such that
$$|N_{B_v}(w_2)\cap U_3| \ge \frac{|B_v|/2}{|U_2|} > \frac{300\epsilon^{1/12} \tilde{n}^2/2}{\tilde{n}/5} > 600\epsilon^{1/12} \tilde{n}.$$

Let $U_3' = N_{B_v}(w_2)\cap U_3$, and $U_j' = N(v) \cap U_j$ for $j \in \{4,5,6\}$.
Since $|U_3'| \ge 600\epsilon^{1/12} \tilde{n}$ and $|U_j'| \ge \tilde{n}/13$ for $j \in \{4,5,6\}$,
applying Lemma~\ref{LEMMA:greedily-embedding-Gi} with $T = \{w_2\}$, $S = \{1,3,4,5,6\}$, and $\eta = 30\epsilon^{1/4}$
we obtain $w_1 \in U_1$ and $w_j \in U_j'$ for $j\in \{3,4,5,6\}$ (see Figure~\ref{figure:Bv-small}) such that
the induced subgraph of $\widetilde{\mathcal{H}}_i$ on $\{w_1,\ldots,w_6\}$ is a copy of $\mathcal{G}_{6}^{2}$.
For $j\in \{4,5,6\}$ let $e_j \in \widetilde{\mathcal{H}}_i$ be an edge containing $v$ and $w_j$.
Let $F = \widetilde{\mathcal{H}}_i[\{w_1,\ldots,w_6\}] \cup \{e_j\colon j \in\{4,5,6\}\} \cup \{vw_2w_3\}$.
It is easy to see that $F$ is a member in $\mathcal{K}_{6}^{3}$ with core $\{v,w_2,\ldots,w_6\}$ (see Figure~\ref{figure:Bv-small}).
So, by assumption, either $F\subset \mathcal{G}_{m}^{1}$ or $F\subset \mathcal{G}_{m}^{2}$ for any integer $m$.
It is easy to see that the former case cannot hold since the induced subgraph of $F$ on the set $\{w_1,\ldots,w_6\}$
is a copy of $\mathcal{G}_{6}^{2}$ and $\mathcal{G}_{6}^{2} \not\subset \mathcal{G}_{m}^{1}$ for any integer $m$.
So, $F\subset \mathcal{G}_{n}^{2}$ for some integer $m$.
In other words, there exists a map $\psi\colon V(F) \to V(\mathcal{G}_{6}^{2})$ such that
$\psi(e)\in \mathcal{G}_{6}^{2}$ for all $e\in F$.
Notice that both $\{w_1,\ldots,w_6\}$ and $\{v,w_2,\ldots,w_6\}$ are $2$-covered in $F$,
so the restrictions of $\psi$ on $\{w_1,\ldots,w_6\}$ and $\{v,w_2,\ldots,w_6\}$ are both injective (similar to the proof of Lemma~\ref{lemma-M=homM}),
and moreover, $\psi(v) = \psi(w_1)$.
Let $w = \psi(v) = \psi(w_1)$.
Notice that the induced subgraph of $L_{F}(w_1)$ on $\{w_2,\ldots,w_3\}$ has size $8$
and $w_2w_3\in L_{F}(v)\setminus L_{F}(w_1)$.
Since $\psi$ preserves edges, the degree of $w$ in $\mathcal{G}_{6}^{2}$ should be at least $8+1 = 9$.
However, this contradicts the fact that the maximum degree of $\mathcal{G}_{6}^{2}$ is $8$.
\end{proof}

A consequence of Claim~\ref{CLAIM:bad-edges-are-few} is that the size of $M_v$ satisfies
\begin{align}
|M_v| = |\widetilde{G}_{1} \setminus L(v)|
& = |\widetilde{G}_{1}| - |\widetilde{G}_{1} \cap L(v)| \notag \\
& = |\widetilde{G}_{1}| - \left(|L(v)| - |B_v| \right) \notag \\
& < 8 \left( {1}/{6} + 10 \epsilon^{1/4} \right)^2 \tilde{n}^2
- \left( \left({2}/{9} -10 {\epsilon}^{1/2}\right) \tilde{n}^2 - |B_v| \right)
 < 400\epsilon^{1/12}\tilde{n}^2. \notag
\end{align}

\begin{figure}[htbp]
\centering
\begin{tikzpicture}[xscale=3,yscale=3]
     \node (v) at (-0.85,0.25) {};
    \node (w6) at (-1,0) {};
    \node (u2) at (0.5,0.5) {};
    \node (u3) at (1,0) {};
    \node (w4) at (0.5,-0.5) {};
    \node (w5) at (-0.5,-0.5) {};
    \node (w1) at (-0.5,0.5) {};
    \node (w2) at (0.63,0.5) {};
    \node (w3) at (1.13,0) {};

    \fill (v) circle (0.02) node [above] {$v$};
    \fill (w6) circle (0.02) node [left] {$w_6$};
    \fill (u2) circle (0.02) node [above] {$u_2$};
    \fill (u3) circle (0.02) node [above] {$u_3$};
    \fill (w4) circle (0.02) node [below] {$w_4$};
    \fill (w5) circle (0.02) node [below] {$w_5$};
    \fill (w1) circle (0.02) node [above] {$w_1$};
    \fill (w2) circle (0.02) node [right] {$w_2$};
    \fill (w3) circle (0.02) node [right] {$w_3$};

    \draw[line width=0.8pt] (0.5,0.5) circle [radius = 0.18];
    \draw[line width=0.8pt] (0.5,-0.5) circle [radius = 0.18];
    \draw[line width=0.8pt] (-0.5,0.5) circle [radius = 0.18];
    \draw[line width=0.8pt] (-0.5,-0.5) circle [radius = 0.18];
    \draw[line width=0.8pt] (1,0) circle [radius = 0.18];
    \draw[line width=0.8pt] (-1,0) circle [radius = 0.18];

    \draw[line width=0.5pt] (-0.85, 0.25) -- (0.5,0.5) -- (-1,0) -- (-0.85,0.25);
    \draw[line width=0.5pt] (1,0) -- (0.5,-0.5) -- (-0.5,-0.5) -- (1,0);
    \draw[line width=0.5pt] (-0.85, 0.25) -- (1,0);
    \draw[line width=0.5pt] (-0.85, 0.25) -- (0.5,-0.5);
    \draw[line width=0.5pt] (-0.85, 0.25) -- (-0.5,-0.5);
    \draw[line width=0.5pt] (0.5, 0.5) -- (1,0);
    \draw[line width=0.5pt] (0.5, 0.5) -- (0.5,-0.5);
    \draw[line width=0.5pt] (0.5, 0.5) -- (-0.5,-0.5);
    \draw[line width=0.5pt] (-1,0) -- (1,0);
    \draw[line width=0.5pt] (-1,0) -- (0.5,-0.5);
    \draw[line width=0.5pt] (-1,0) -- (-0.5,-0.5);

    \node at (-0.5,0.65+0.1) {$U_1$};
    \node at (0.5,0.65+0.1) {$U_2$};
    \node at (0.5,-0.65-0.1) {$U_4$};
    \node at (-0.5,-0.65-0.1) {$U_5$};
    \node at (-1,0-0.25) {$U_6$};
    \node at (1,0-0.25) {$U_3$};
\end{tikzpicture}
\caption{The $3$-graph $F = \widetilde{\mathcal{H}}_i[\{v,u_2,u_3,w_1,\ldots,w_6\}] \cup \{vu_2u_3\} \cup \{e_{u_3w_4}\}$
is a member in $M_3$ with core $\{v,u_2,u_3,w_4,w_5,w_6\}$.}
\label{figure:Bv-empty}
\end{figure}
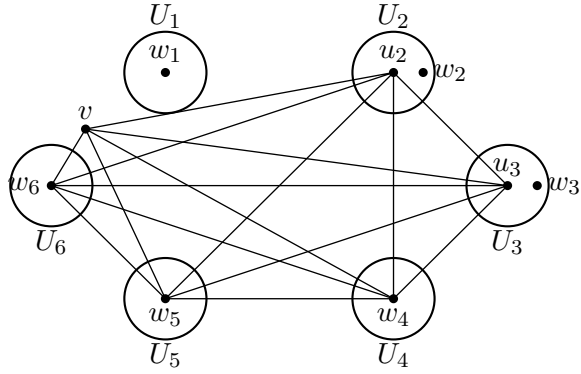

\begin{claim}\label{CLAIM:bad-edges-are-0}
We have $B_v = \emptyset$. In other words, $L_{\widetilde{\mathcal{H}}_i}(v) \subset \widetilde{G}_{1}$.
\end{claim}
\begin{proof}[Proof of Claim~\ref{CLAIM:bad-edges-are-0}]
Suppose to the contrary that there exists an edge $u_2u_3\in B_v$
and by symmetry we may assume that $u_2\in U_2$ and $u_3\in U_3$.
For $j\in\{4,5,6\}$ let $U_j' = U_j \cap N(v) \cap N(u_1)\cap N(u_2)$ and
notice that due to $|M_v| \le 400\epsilon^{1/12}\tilde{n}^2$ and Claim~\ref{claim-P-size-2}
we have $|U_j'| \ge |U_j|/2 > \tilde{n}/20$.
Applying Lemma~\ref{LEMMA:greedily-embedding-Gi} with $T=\{u_2,u_3\}$, $S = [6]$, and $\eta = 400\epsilon^{1/36}$
we obtain $w_j \in U_j'$ for $j \in [6]$ (see Figure~\ref{figure:Bv-empty}) such that
\begin{itemize}
\item[(a)]
$\widetilde{\mathcal{H}}_i[\{w_1,\ldots,w_6\}] \cong \mathcal{G}_{6}^{2}$,
\item[(b)]
$L_{\widetilde{\mathcal{H}}_i}(v)[\{w_2,\ldots,w_6\}] = L_{\widehat{\mathcal{G}}}(w_1)[\{w_2,\ldots,w_6\}]$,
\item[(c)]
$L_{\widetilde{\mathcal{H}}_i}(u_2)[\{w_1,w_3,\ldots,w_6\}] = L_{\widehat{\mathcal{G}}}(u_2)[\{w_1,w_3,\ldots,w_6\}]$, and
\item[(d)]
$L_{\widetilde{\mathcal{H}}_i}(u_3)[\{w_1,w_2,w_4,w_5,w_6\}] = L_{\widehat{\mathcal{G}}}(u_3)[\{w_1,w_2,w_4,w_5,w_6\}]$.
\end{itemize}
Let $e_{u_3w_4} \in \widetilde{\mathcal{H}}_i$ be an edge containing $u_3$ and $w_4$.
Let $F = \widetilde{\mathcal{H}}_i[\{v,u_2,u_3,w_1,\ldots,w_6\}] \cup \{vu_2u_3\} \cup \{e_{u_3w_4}\}$.
Then it is easy to see that $F$ is a member in $\mathcal{K}_{6}^{3}$ with core $\{v,u_2,u_3,w_4,w_5,w_6\}$ (see Figure~\ref{figure:Bv-empty}).
Similar to the proof of Claim~\ref{CLAIM:bad-edges-are-few},  $F\subset \mathcal{G}_{m}^{2}$ for some integer $m$.
In other words, there exists a map $\psi\colon V(F) \to V(\mathcal{G}_{6}^{2})$ such that
$\psi(e)\in \mathcal{G}_{6}^{2}$ for all $e\in F$.
Notice that both $\{w_1,\ldots,w_6\}$ and $\{v,u_2,u_3,w_4,w_5,w_6\}$ are $2$-covered in $F$,
so the restrictions of $\psi$ on sets $\{w_1,\ldots,w_6\}$ and $\{v,u_2,u_3,w_4,w_5,w_6\}$ are both injective (similar to the proof of Lemma~\ref{lemma-M=homM}), and
moreover, $\psi(v) = \psi(w_1)$ (due to (b), $v$ is adjacent to all vertices in $\{w_2,\ldots,w_6\}$, so $\psi(v)$ is distinct from $\{\psi(w_2),\ldots,\psi(w_6)\}$), $\psi(u_2) = \psi(w_2)$ (due to (c) and a similar reason), and $\psi(u_3) = \psi(w_3)$ (due to (d) and a similar reason).
Let $w = \psi(v) = \psi(w_1)$.
Notice that the induced subgraph of $L_{F}(w_1)$ on $\{w_2,\ldots,w_6\}$ has size $8$
and $u_2u_3\in L_{F}(v)\setminus L_{F}(w_1)$.
Since $\psi$ preserves edges, the degree of $w$ in $\mathcal{G}_{6}^{2}$ should be at least $8+1 = 9$.
However, this contradicts the fact that the maximum degree of $\mathcal{G}_{6}^{2}$ is $8$.
\end{proof}

Define
\begin{align}
V_j^{i}
=
\begin{cases}
U_{1}\cup C_v, & \text{if } j = 1, \\
U_j, & \text{otherwise}.
\end{cases} \notag
\end{align}
By Claim~\ref{CLAIM:bad-edges-are-0}, $L(v) \subset \widetilde{G}_{1}$.
Therefore, $\widetilde{\mathcal{H}}_{i}$ is $\mathcal{G}^2_{6}$-colorable with set of parts $\{V_{1}^{i},\ldots,V_{6}^{i}\}$.
This completes the proof of Lemma~\ref{main-lemma-Ht-G2-color}.
\end{proof}

\section{Feasible Region of $\mathcal{M}$}\label{SEC:feasible-region-2-global-max}
In this section we consider the feasible region $\Omega(\mathcal{M})$ of $\mathcal{M}$ and prove Theorem \ref{thm-2-global-max}.
First, notice that our proof of Theorem \ref{thm-2-stable} actually gives the following stronger statement.

\begin{thm}\label{thm-2-stable-stronger}
For every sufficiently small $\epsilon > 0$ there exists $n_0$ such that the following holds for all $n \ge n_0$.
Every $\mathcal{M}$-free $3$-graph $\mathcal{H}$ with $n$ vertices and at least $2n^3/27 - \epsilon n^3$ edges
contains $W \subset V(\mathcal{H})$ with $|W| > n - 3\epsilon^{1/2}n$ such that
$\delta(\mathcal{H}[W]) \ge {2n^2}/{9} - 20 \epsilon^{1/2}n$ and $\mathcal{H}[W]$ is either semibipartite or $\mathcal{G}^{2}_{6}$-colorable.
\end{thm}

Theorem \ref{thm-2-stable-stronger} gives the following lemma.

\begin{lemma}\label{lemma-shadow-semibi-G2-color}
Let $\epsilon > 0$ be sufficiently small and $n$ (related to $\epsilon$) be sufficiently large.
Suppose that $\mathcal{H}$ is an $\mathcal{M}$-free $3$-graph with $n$ vertices and at least $2n^3/27 - \epsilon n^3$ edges.
Then,
\begin{align}
\text{either }\left||\partial \mathcal{H}| - \frac{5}{12}n^2\right| < 100\epsilon^{1/4}n^2
\text{ or } \left||\partial \mathcal{H}| - \frac{4}{9}n^2\right| < 100\epsilon^{1/4}n^2.\notag
\end{align}
\end{lemma}
\begin{proof}
Let $\epsilon > 0$ be sufficiently small and $n$ (related to $\epsilon$) be sufficiently large.
Let $\mathcal{H}$ be an $\mathcal{M}$-free $3$-graph with $n$ vertices and at least $2n^3/27 - \epsilon n^3$ edges.
By Theorem \ref{thm-2-stable-stronger},
there exists $W \subset V(\mathcal{H})$ with $|W| > n - 3\epsilon^{1/2}n$ such that $\delta(\mathcal{H}[W]) \ge {2n^2}/{9} - 20 \epsilon^{1/2}n^2$
and $\mathcal{H}[W]$ is either semibipartite or $\mathcal{G}^{2}_{6}$-colorable.
Let $Z = V(\mathcal{H})\setminus W$, $\tilde{n} = |W|$, and $\widetilde{\mathcal{H}} = \mathcal{H}[W]$.
Then,
\begin{align}\label{equ-size-H[W]}
|\widetilde{\mathcal{H}}| = \frac{1}{3}\sum_{w \in W}d_{\widetilde{\mathcal{H}}}(w)
> \frac{1}{3} \left( n - 3\epsilon^{1/2}n \right)\left(\frac{2}{9}n^2 - 20 \epsilon^{1/2}n^2 \right)
> \frac{2}{27}n^3 - 20\epsilon^{1/2}n^3.
\end{align}

Suppose that $\mathcal{H}[W]$ is semibipartite and
let $L$ and $R$ denote the two parts of $\mathcal{H}[W]$ such that every $E \in \mathcal{H}[W]$
satisfies $|A \cap E| = 1$ and $|B \cap E| = 2$.
Note from Claim \ref{claim-size-Ai+1-Bi+1} that
\begin{align}\label{equ-size-L-FR}
|L| = \frac{|W|}{3} \pm 4 \epsilon^{1/4} |W| = \frac{n}{3}\pm 8\epsilon^{1/4}n
\end{align}
and
\begin{align}\label{equ-size-R-FR}
|L| = \frac{2|W|}{3} \pm 4 \epsilon^{1/4} |W| = \frac{2n}{3}\pm 8\epsilon^{1/4}n.
\end{align}

First we prove the lower bound for $|\partial {\mathcal{H}}|$.
Let
\[
(\partial\widetilde{\mathcal{H}})[R] = \left\{ uv \in \partial\widetilde{\mathcal{H}}: \{u,v\} \subset R \right\},
\]
and
\[
(\partial\widetilde{\mathcal{H}})[L,R] = \left\{ uv \in \partial\widetilde{\mathcal{H}}: u \in L, v\in R \right\}.
\]
Since $\widetilde{\mathcal{H}}$ is semibipartite,
$$|L||(\partial\widetilde{\mathcal{H}})[R]| \ge \sum_{v \in L}d_{\widetilde{H}}(v) = |\widetilde{H}|$$
and
$$|R||(\partial\widetilde{\mathcal{H}})[L,R]| \ge \sum_{u \in R}d_{\widetilde{H}}(u) = 2|\widetilde{H}|.$$
Consequently,
\begin{align}
|\partial \widetilde{\mathcal{H}}|
 = |(\partial\widetilde{\mathcal{H}})[R]| + |(\partial\widetilde{\mathcal{H}})[L,R]|
& \ge \frac{|\widetilde{\mathcal{H}}|}{|L|} + \frac{2|\widetilde{\mathcal{H}}|}{|R|} \notag \\
& \overset{(\ref{equ-size-H[W]}),(\ref{equ-size-L-FR}),(\ref{equ-size-R-FR})}{>}
\frac{{2n^3}/{27} - 20 \epsilon^{1/2}n^3}{({1}/{3} + 8\epsilon^{1/4})n}
+ \frac{2({2n^3}/{27} - 20 \epsilon^{1/2}n^3)}{({2}/{3} + 8\epsilon^{1/4})n} \notag \\
& > \frac{4}{9}n^2 - 100\epsilon^{1/2}n^2. \notag
\end{align}
Therefore, $|\partial \mathcal{H}| \ge |\partial \widetilde{\mathcal{H}}| > {4n^2}/{9} - 100\epsilon^{1/2}n^2$.

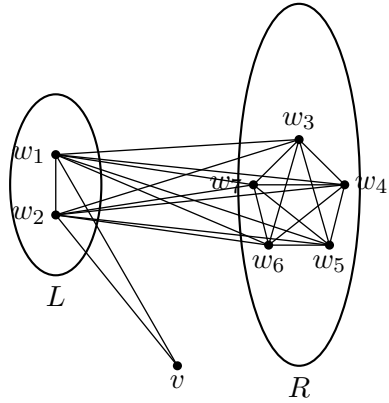
\begin{figure}[htbp]\label{fig-Lv[L]-empty}
\centering
\begin{tikzpicture}[xscale=4,yscale=4]
    \node (v) at (0,-0.6) {};
    \fill (v) circle (0.015) node [below] {$v$};
    \node (w1) at (-0.4,0.1) {};
    \fill (w1) circle (0.015) node [left] {$w_1$};
    \node (w2) at (-0.4,-0.1) {};
    \fill (w2) circle (0.015) node [left] {$w_2$};
    \node (w3) at (0.4,0.15) {};
    \fill (w3) circle (0.015) node [above] {$w_3$};
    \node (w5) at (0.5,-0.2) {};
    \fill (w5) circle (0.015) node [below] {$w_5$};
    \node (w6) at (0.3,-0.2) {};
    \fill (w6) circle (0.015) node [below] {$w_6$};
    \node (w4) at (0.55,0) {};
    \fill (w4) circle (0.015) node [right] {$w_4$};
    \node (w7) at (0.25,0) {};
    \fill (w7) circle (0.015) node [left] {$w_7$};
    \draw[rotate=0,line width=0.8pt] (-0.4,0) ellipse [x radius=0.15, y radius=0.3];
    \draw[rotate=0,line width=0.8pt] (0.4,0) ellipse [x radius=0.2, y radius=0.6];
    \draw[line width=0.5pt] (0,-0.6) -- (-0.4,-0.1);
    \draw[line width=0.5pt] (0,-0.6) -- (-0.4,0.1);
    \draw[line width=0.5pt] (0.3,-0.2) -- (-0.4,0.1);
    \draw[line width=0.5pt] (0.3,-0.2) -- (-0.4,-0.1);
    \draw[line width=0.5pt] (0.3,-0.2) -- (0.4,0.15);
    \draw[line width=0.5pt] (0.3,-0.2) -- (0.5,-0.2);
    \draw[line width=0.5pt] (0.3,-0.2) -- (0.25,0);
    \draw[line width=0.5pt] (0.3,-0.2) -- (0.55,0);
    \draw[line width=0.5pt] (-0.4,-0.1) -- (-0.4,0.1);
    \draw[line width=0.5pt] (0.4,0.15) -- (0.5,-0.2);
    \draw[line width=0.5pt] (0.4,0.15) -- (0.25,0);
    \draw[line width=0.5pt] (0.4,0.15) -- (0.55,0);
    \draw[line width=0.5pt] (0.5,-0.2) -- (0.25,0);
    \draw[line width=0.5pt] (0.5,-0.2) -- (0.55,0);
    \draw[line width=0.5pt] (0.55,0) -- (0.25,0);
    \draw[line width=0.5pt] (-0.4,0.1) -- (0.25,0);
    \draw[line width=0.5pt] (-0.4,0.1) -- (0.55,0);
    \draw[line width=0.5pt] (-0.4,0.1) -- (0.5,-0.2);
    \draw[line width=0.5pt] (-0.4,0.1) -- (0.4,0.15);
    \draw[line width=0.5pt] (-0.4,-0.1) -- (0.25,0);
    \draw[line width=0.5pt] (-0.4,-0.1) -- (0.55,0);
    \draw[line width=0.5pt] (-0.4,-0.1) -- (0.5,-0.2);
    \draw[line width=0.5pt] (-0.4,-0.1) -- (0.4,0.15);
    \node at (-0.4,-0.3-0.07) {$L$};
    \node at (0.4,-0.6-0.07) {$R$};
\end{tikzpicture}
\caption{The $3$-graph $F = \mathcal{H}[\{w_1,\ldots,w_7\}] \cup \{uw_1w_2\}$ is a member in $M_2$ with core $\{w_1,\ldots,w_7\}$.
In particular, $\tau\{w_1w_3w_4,w_2w_5w_6\} > 1$.}
\label{fig:semibipartite-no-link-in-L}
\end{figure}

Next, we prove the upper bound for $|\partial {\mathcal{H}}|$.
Let $v \in Z$ and suppose that there exists an edge $w_1w_2 \in L_{\mathcal{H}}(v)[L]$.
Since $L_{\mathcal{H}[W]}(w_1)$ and $L_{\mathcal{H}[W]}(w_2)$ are graphs on $R$ and
$|L_{\mathcal{H}[W]}(w_1)|, |L_{\mathcal{H}[W]}(w_2)| \ge 2n^2/9-20\epsilon^{1/2}n^2$,
the edge density of $L_{\mathcal{H}[W]}(w_1) \cap L_{\mathcal{H}[W]}(w_2)$ is close to $1$.
So, by Tur\'{a}n's theorem, there exists a copy of $K_5$ in $L_{\mathcal{H}[W]}(w_1) \cap L_{\mathcal{H}[W]}(w_2)$.
We may assume that the vertex set of this $K_5$ is $\{w_3, w_4, w_5, w_6, w_7\}$ (see Figure~\ref{fig:semibipartite-no-link-in-L}).
Let $F = \mathcal{H}[\{w_1,\ldots,w_7\}] \cup \{uw_1w_2\}$.
Then it is easy to see that $F$ is a member in $M_2$ with core $\{w_1,\ldots,w_7\}$, a contradiction.
Therefore, $L_{\mathcal{H}}(v)[L] = \emptyset$ for all $v \in Z$, and it follows that
\begin{align}
|\partial \mathcal{H}|
& \le \binom{|W|}{2} - \binom{|L|}{2} +  |Z||W| + \binom{|Z|}{2} \notag \\
& \le \binom{n}{2} - \binom{|L|}{2}
 < \frac{1}{2} n^2 - \frac{1}{2}\left( \frac{1}{3} - 8\epsilon^{1/4}\right)n^2
 < \frac{4}{9}n^2 + 100\epsilon^{1/4}n^2. \notag
\end{align}
Therefore, if $\widetilde{\mathcal{H}}$ is semibipartite, then
\begin{align}
\frac{4}{9}n^2 - 100\epsilon^{1/2}n^2 < |\partial \mathcal{H}| < \frac{4}{9}n^2 + 100\epsilon^{1/4}n^2. \notag
\end{align}

Suppose that $\widetilde{\mathcal{H}}$ is $\mathcal{G}^{2}_{6}$-colorable and
let $\mathcal{P}:=\{A_1,A_2,A_3,A_4,B_1,B_2\}$ be the set of six parts of $\widetilde{\mathcal{H}}$ such that
there is no edge between $A_1A_2B_1$, $A_1A_2B_2$, $A_3A_4B_1$, and $A_3A_4B_2$.
Notice from Claim \ref{claim-P-size-2} that
\begin{align}\label{equ-size-S-FR}
|S| = \frac{|W|}{6}\pm 10 \epsilon^{1/2} |W| = \frac{n}{6} \pm 10 \epsilon^{1/4} n, \text{ for all }S \in \mathcal{P}.
\end{align}

First, we prove the lower bound for $|\partial {\mathcal{H}}|$.
Since $\widetilde{\mathcal{H}}$ is $\mathcal{G}^{2}_{6}$-colorable, $\partial \widetilde{\mathcal{H}}$ is a $6$-partite graph
the set of six parts $\mathcal{P}$.
Let $\mathcal{G}$ denote the blow up of $\mathcal{G}^{2}_{6}$ with the set of six parts $\mathcal{P}$ such that
there is no edge between $A_1A_2B_1$, $A_1A_2B_2$, $A_3A_4B_1$, and $A_3A_4B_2$.
Notice that for every $e \in \partial\mathcal{G}\setminus \partial \widetilde{\mathcal{H}}$ 
there are at least $2(|W|/6 - 10 \epsilon^{1/2} |W|)$
sets $E \in \mathcal{G}\setminus \widetilde{\mathcal{H}}$ such that $e \in E$.
Therefore,
\begin{align}
|\partial\mathcal{G}\setminus \partial \widetilde{\mathcal{H}}|
& \le \frac{3|\mathcal{G}\setminus \widetilde{\mathcal{H}}|}{2(|W|/6 - 10 \epsilon^{1/2} |W|)}
 \overset{(\ref{equ-size-H[W]}),(\ref{equ-size-S-FR})}{<}
\frac{3\times 20 \epsilon^{1/2} n^2}{2({n}/{6} - 10 \epsilon^{1/4} n)}
< 400 \epsilon^{1/2} n^2, \notag
\end{align}
and it follows that
\begin{align}
|\partial \widetilde{\mathcal{H}}|
> |\partial\mathcal{G}| - 400 \epsilon^{1/2} n^2
\overset{(\ref{equ-size-S-FR})}{>}
\binom{6}{2}\left( \frac{n}{6} - 10 \epsilon^{1/4} n \right)^2 -  400 \epsilon^{1/2} n^2
> \frac{5}{12}n^2 - 100 \epsilon^{1/4} n^2 \notag
\end{align}
Therefore, $|\partial \mathcal{H}| \ge |\partial \widetilde{\mathcal{H}}| > {5n^2}/{12} - 100\epsilon^{1/4}n^2$.

Next, we prove the upper bound for $|\partial {\mathcal{H}}|$.
Let $v \in Z$ and suppose that $L_{\mathcal{H}}(v)[S] \neq \emptyset$ for some $S \in \mathcal{P}$.
Then Claim \ref{claim-Lv[S]-empty}\footnote{Even though Claim \ref{claim-Lv[S]-empty} was proved only for vertices in $W$, but in fact, 
its proof does not require $v$ to have a large degree. So it also holds for vertices in $Z$.}
 implies that $\mathcal{H}$ contains a copy of a $3$-graph in $M_2$, a contradiction.
Therefore, $L_{\mathcal{H}}(v)[S] = \emptyset$ for all $S \in \mathcal{P}$, and it follows that
\begin{align}
|\partial \mathcal{H}|
\le \frac{5}{12}|W|^2 + |Z||W| + \binom{|Z|}{2}
< \frac{5}{12}n^2 + 100\epsilon^{1/4}n^2. \notag
\end{align}
Therefore, if $\widetilde{\mathcal{H}}$ is $\mathcal{G}^{2}$-colorable, then
\begin{align}
\frac{5}{12}n^2 - 100\epsilon^{1/4}n^2 < |\partial \mathcal{H}| < \frac{5}{12}n^2 + 100\epsilon^{1/4}n^2. \notag
\end{align}
\end{proof}

Now we are ready to prove Theorem \ref{thm-2-global-max}.

\begin{proof}[Proof of Theorem \ref{thm-2-global-max}]
Let
\[
\mathcal{S}_n = \left\{ A \subset [n]: 1 \in A \right\}.
\]
Since $\mathcal{S}_n$ is $\mathcal{M}$-free and $|\partial \mathcal{S}_n| = \binom{n}{2}$,
it follows from Observation 1.5 in \cite{LM19A} that ${\rm proj}\Omega(\mathcal{M}) = [0,1]$.
On the other hand, it follows from Theorem \ref{thm-Turan-num-M} that $g(\mathcal{M},x) \le 4/9$ for all $x \in [0,1]$ and
$g(\mathcal{M},5/6) = g(\mathcal{M},8/9) = 4/9$.

Now suppose that $\left( \mathcal{H}_{k} \right)_{k=1}^{\infty}$ is a sequence of $\mathcal{M}$-free $3$-graphs
with $\lim_{k \to \infty}v(\mathcal{H}_{k}) = \infty$, $\lim_{k \to\infty}d(\partial \mathcal{H}_k) = x_0$, and
$\lim_{k \to\infty}d(\mathcal{H}_k) = 4/9$.
For any sufficiently small $\epsilon > 0$ and sufficiently $n_0$,
there exists $k_0$ such that $v(\mathcal{H}_{k}) \ge n_0$
and $|\mathcal{H}_{k}| > 2(v(\mathcal{H}_{k}))^3/27 - \epsilon (v(\mathcal{H}_{k}))^3$ for all $k \ge k_0$.
Therefore, by Lemma \ref{lemma-shadow-semibi-G2-color}, for every $k \ge k_0$
either
\[
\frac{8}{9} - 200\epsilon^{1/4} < \frac{|\partial \mathcal{H}_k|}{\binom{v(\mathcal{H}_{k})}{2}} < \frac{8}{9} + 200\epsilon^{1/4}
\]
or
\[
\frac{5}{6} - 200\epsilon^{1/4} < \frac{|\partial \mathcal{H}_k|}{\binom{v(\mathcal{H}_{k})}{2}} < \frac{5}{6} + 200\epsilon^{1/4}.
\]
Letting $\epsilon \to 0$ we obtain either $x_0 = 8/9$ or $x_0 = 5/6$, and this completes the proof.
\end{proof}

\section{Open Problems}\label{SEC:open-problem}
Recall from Definitions \ref{dfn-t-stable} and \ref{dfn-stability-number} that the stability number $\xi(\mathcal{F})$ of a family $\mathcal{F}$
is the minimum integer $t$ such that $\mathcal{F}$ is $t$-stable.
We constructed a family $\mathcal{M}$ of $3$-graphs and proved that $\xi(\mathcal{M}) = 2$.
It seems natural to consider the following problems.

\begin{prob}\footnote{This problem was recently solved by Reiher and the authors in \cite{LMR1,LMR3}.}
For every pair $(r, t)$ with $r \ge 3$ and $t \ge 2$,
determine if there exists a family $\mathcal{F}$ of $r$-graphs with $\xi(\mathcal{F}) = t$.
\end{prob}

Proposition~\ref{prop-stability-number-K43} states that $\xi(K_{4}^{3}) = \infty$ (assuming that Tur\'{a}n's conjecture is true).
However, determining $\textrm{ex}(n,K_4^{3})$ still seems far beyond reach.
So we pose the following problem.
\begin{prob}
Determine $\textrm{ex}(n,\mathcal{F})$ for some family $\mathcal{F}$ with $\xi(\mathcal{F}) = \infty$.
\end{prob}

\section{Acknowledgment}
We are very grateful to all the referees for their many helpful comments. In particular, for the suggestion of using the result in~\cite{FPS05} which substantially shortened the presentation  and for the cleaner and shorter proofs of some technical statements (Lemma~\ref{lemma-lagrang-G26} and Claim~\ref{claim-P-size-2}).
\bibliographystyle{abbrv}
\bibliography{twostable}
\end{document}